\documentclass[12pt]{amsart}
\usepackage{amssymb,latexsym,pstricks,amsmath,amsthm,amsfonts, enumerate}
\usepackage{color}
\usepackage{auto-pst-pdf}
\usepackage{tikz}

\numberwithin{equation}{section}

\newtheorem{theorem}{Theorem}[section]
\newtheorem{proposition}[theorem]{Proposition}

\newtheorem{corollary}[theorem]{Corollary}
\newtheorem{lemma}[theorem]{Lemma}

\theoremstyle{definition}
\newtheorem{remark}[theorem]{Remark}
\newtheorem{example}[theorem]{Example}
\newtheorem{definition}[theorem]{Definition}

\let\oldmarginpar\marginpar
\renewcommand\marginpar[1]{\-\oldmarginpar[\raggedleft\small\sf
#1]{\raggedright\small\sf #1}}

\newcommand{\myAA}{\mathcal{A}}

\newcommand{\FF}{\mathcal{F}}

\newcommand{\ZZ}{\mathbb{Z}}

\newcommand{\RR}{\mathbb{R}}
\newcommand{\QQ}{\mathbb{Q}}

\newcommand{\mat}[4]{\left(\!\!\begin{array}{cc}
#1 & #2 \\ #3 & #4 \\
\end{array}\!\!\right)}
\newcommand{\sh}{{rsh}}
\newcommand{\SH}{{sh}}

\def\Tiny{ \font\Tinyfont = cmr10 at 7pt \relax  \Tinyfont}
\newcommand*\Tinycircled[1]{\tikz[baseline=(char.base)]{
            \node[shape=circle,draw,inner sep=.5pt] (char) {\Tiny #1};}}
\newcommand*\circled[1]{\tikz[baseline=(char.base)]{
            \node[shape=circle,draw,inner sep=.5pt] (char) {#1};}}

\begin{document}

\title[Greedy elements in rank 2 cluster algebras]
{Greedy elements in rank 2 cluster algebras}

\author{Kyungyong Lee}
\address{\noindent Department of Mathematics, Wayne State University, Detroit, MI 48202, USA}
\email{klee@math.wayne.edu}

\author{Li Li}
\address{\noindent Department of Mathematics and Statistics, Oakland University, Rochester, MI 48309, USA}
\email{li2345@oakland.edu}

\author{Andrei Zelevinsky}
\address{\noindent Department of Mathematics, Northeastern University,
Boston, MA 02115, USA}
\email{andrei@neu.edu}


\date{August 25, 2012; revised November 5, 2012}

\thanks{Research 
supported in part by NSF grants DMS-0901367 (K.~L.) and DMS-1103813 (A.~Z.)}

\begin{abstract}
 A lot of recent activity in the theory of cluster algebras has been directed towards various constructions of ``natural" bases in them. One of the approaches to this problem was developed several years ago by P.~Sherman - A.~Zelevinsky who have shown that the indecomposable positive elements form an integer basis in any rank 2 cluster algebra of finite or affine type. It is strongly suspected (but not proved) that this property does not extend beyond affine types. Here we go around this difficulty by constructing a new basis in any rank 2 cluster algebra that we call the greedy basis. It consists of a special family of indecomposable positive elements that we call greedy elements. Inspired by a recent work of K.~Lee - R.~Schiffler and D.~Rupel, we give explicit combinatorial expressions for greedy elements using the language of Dyck paths.
\end{abstract}

\maketitle

\bigskip


\section{Introduction and main results}
\label{sec:intro}

The original motivation for the study of cluster algebras initiated
in~\cite{fz-ClusterI} was to design an
algebraic framework for understanding \emph{total positivity} and
\emph{canonical bases} associated by G.~Lusztig to any semisimple algebraic group. 
A lot of recent activity in the field has been directed towards various constructions of ``natural" bases in cluster algebras.
An overview of these approaches with relevant references can be found in \cite{plamondon}.

This paper builds upon the approach developed in \cite{sz-Finite-Affine}, where it was shown that the indecomposable positive elements 
form a $\ZZ$-basis in any rank 2 cluster algebra of finite or affine type (the definitions will be recalled in a moment). 
The authors of \cite{sz-Finite-Affine} have suspected that this property does \emph{not} extend beyond affine types (we share this suspicion although are still unable to confirm it decisively). 
In an unpublished follow-up to \cite{sz-Finite-Affine} they have introduced a special family of \emph{greedy elements} in (the completion of) an arbitrary rank~2 cluster algebra $\myAA$, and made several conjectures about them, including the claim that all these elements are indecomposable positive elements, and that they form a $\ZZ$-basis in $\myAA$. 

This paper is devoted to the study of greedy elements. 
In particular, we prove all the conjectures mentioned above. 
The key new ingredient is an explicit combinatorial expression for greedy elements inspired by an expression for cluster variables given in 
\cite{ls-comm,ls-noncomm,r-noncomm}.

Now we introduce our setup and state our main results.  
Let $\FF = \QQ(x_1,x_2)$ be the field of rational functions in two
(commuting) independent variables $x_1$ and $x_2$ with rational
coefficients. Given positive integers $b$ and $c$, recursively
define elements $x_m \in \FF$ for $m \in \ZZ$ by the relations
\begin{equation}
\label{eq:clusterrelations}
x_{m-1} x_{m+1} =
\left\{
\begin{array}[h]{ll}
x_m^b + 1 & \quad \mbox{for  $m$ odd;} \\
x_m^c + 1 & \quad \mbox{for $m$ even.}
\end{array}
\right.
\end{equation}
The (coefficient-free) cluster algebra $\myAA=\myAA(b,c)$ is, by definition the
subring of $\FF$ generated by the $x_m$ for all $m \in \ZZ$. The
elements $x_m$ are called \emph{cluster variables} and the
relations (\ref{eq:clusterrelations}) are called the
\emph{exchange relations}. The sets $\{x_m,x_{m+1}\}$ for
$m\in\ZZ$ are called \emph{clusters}, and an element of the form $x_m^{d_1} x_{m+1}^{d_2}$ with
$d_1, d_2 \geq 0$ is called a \emph{cluster monomial} at a cluster $\{x_m,x_{m+1}\}$.

It is clear from (\ref{eq:clusterrelations}) that every cluster of
$\myAA$ is a free system of generators of the ambient field $\FF$, so
for every $m \in \ZZ$, each element of $\myAA$ is uniquely expressed as a rational
function in $x_m$ and $x_{m+1}$. According
to the \emph{Laurent phenomenon} established in
\cite{fz-ClusterI,fz-Laurent}, all these rational functions are
actually Laurent polynomials with integer coefficients. The
following stronger result is a special case of the results
in~\cite{fz-ClusterIII}:
\begin{equation}
\label{eq:upper} 
\myAA = \bigcap_{m \in \ZZ} \ZZ[x_m^{\pm 1},
x_{m+1}^{\pm 1}] = \bigcap_{m = 0}^2 \ZZ[x_m^{\pm 1}, x_{m+1}^{\pm
1}],
\end{equation}
where $\ZZ[x_m^{\pm 1}, x_{m+1}^{\pm 1}]$ denotes the ring of
Laurent polynomials with integer coefficients in $x_m$ and
$x_{m+1}$.    The symmetry of the exchange relations
(\ref{eq:clusterrelations}) allows the second intersection in
(\ref{eq:upper}) to be taken over any three consecutive clusters.

We say that a non-zero element $x\in\myAA$ is
\emph{positive at a cluster $\{x_m,x_{m+1}\}$} if all the coefficients in
the expansion of $x$ as a Laurent polynomial in $x_m$ and
$x_{m+1}$ are positive.
We say that $x \in\myAA$ is \emph{positive} if it is positive at all the clusters. 
Thus the set of positive elements in $\myAA$ is equal to $\myAA_+ - \{0\}$, where 
\begin{equation}
\label{eq:upper-positive} \myAA_+ = \bigcap_{m \in \ZZ} \ZZ_{\geq 0}[x_m^{\pm 1},
x_{m+1}^{\pm 1}] \ .
\end{equation}
Clearly, $\myAA_+$ is a \emph{semiring}, i.e., it is closed under addition and multiplication.
We are interested in the additive structure of $\myAA_+$; following \cite{sz-Finite-Affine}, we introduce the following 
important definition. 

\begin{definition}
\label{de:positive-indecomposables} 
A positive element $x\in\myAA$ is
\emph{indecomposable} if it cannot be expressed as the sum of two positive elements. 
\end{definition}

Recall that $\myAA(b,c)$ is of \emph{finite} (resp. \emph{affine}) type if $bc \leq 3$ (resp. $bc = 4$). 
One of the main results of \cite{sz-Finite-Affine} is the following: if $\myAA = \myAA(b,c)$ is of finite or affine type then indecomposable positive elements form 
a $\ZZ$-basis in $\myAA$, and this basis contains all cluster monomials.  
However in the ``wild case" $bc \geq 5$ the situation becomes much more complicated; in particular, we expect the set of indecomposable positive elements to be linearly dependent. 

The main difficulty in studying positive elements stems from the fact that in general they do not allow a ``local" definition. 
Namely, the last equality in \eqref{eq:upper} makes it very easy to check whether a given element of $\FF$ belongs to $\myAA$. 
In contrast to this, it was shown in
\cite[Remark 5.8]{sz-Finite-Affine} that already in the case $b=c=2$ there exist nonpositive
elements of $\myAA$ that are positive at any given finite set of clusters.

To deal with this difficulty we restrict our attention to a special family of elements of $\myAA$. 

\begin{definition}
\label{df:pointed}
An element $x \in \myAA(b,c)$ is pointed at $(a_1, a_2) \in \ZZ^2$ if it has the form 
\begin{equation}
\label{eq:pointed-expansion}
x=x_1^{-a_1}  x_2^{-a_2} \sum_{p,q \geq 0} c(p,q) x_1^{bp} x_2^{cq}	
\end{equation}
with $c(p,q) \in \ZZ$ for all $p$ and $q$, and $c(0,0)=1$. 
\end{definition}

This definition is motivated by the results of \cite{sz-Finite-Affine} where it was shown that, for $bc \leq 4$, every indecomposable positive element 
in $\myAA(b,c)$ is  pointed at some $(a_1, a_2) \in \ZZ^2$. 

Now we are ready to introduce our main object of study, viz. \emph{greedy elements}.
In the following definition and throughout the paper, we use the conventions that the binomial coefficient $\binom{a}{k} $ is
zero unless $0 \leq k \leq a$,  and an empty sum is $0$.

\begin{definition}
\label{df:greedy}
An element $x \in \myAA$ is greedy at $(a_1, a_2) \in \ZZ^2$ if it is pointed at $(a_1, a_2)$, 
and the coefficients $c(p,q)$ in the expansion \eqref{eq:pointed-expansion} satisfy the recurrence relation
\begin{equation}
\label{eq:c-recurrence} 
\aligned
c(p,q)= \max &\left( \sum_{k=1}^p (-1)^{k-1}
c(p-k,q) \binom{a_2\!-\!cq\!+\!k\!-\!1}{k}, \right.\\
&\left.\quad \sum_{k=1}^q
(-1)^{k-1} c(p,q-k) \binom{a_1\!-\!bp\!+\!k\!-\!1}{k}\right)\\
\endaligned
\end{equation}
for every nonzero pair of indices $(p,q) \in \ZZ_{\geq 0}^2$.
\end{definition}

\begin{remark}
\label{rem:greedy}
(a) It is clear that, for a given $(a_1,a_2)$ the relation \eqref{eq:c-recurrence} 
determines $x$ uniquely. 
Thus we can and will use the notation $x = x[a_1,a_2]$. 
In particular, if both $a_1$ and $a_2$ are nonpositive, then $x[a_1,a_2] = x_1^{-a_1} x_2^{-a_2}$, since in this case
all the binomial coefficients appearing in \eqref{eq:c-recurrence} are equal to $0$.  
Thus, every cluster monomial in the initial cluster $\{x_1,x_2\}$ is a greedy element. 
If exactly one of $a_1$ and $a_2$ is nonpositive, it is not hard to show that $x[a_1,a_2]$ is given by \eqref{eq:2-3-quarter} below.  
%
%
However if both $a_1$ and $a_2$ are positive then the existence of $x[a_1,a_2]$ is much less trivial: one has to show
that only finitely many of the coefficients $c(p,q)$ determined by \eqref{eq:c-recurrence} are nonzero (so that $x[a_1,a_2]$
is indeed a Laurent polynomial in $x_1$ and $x_2$), and that $x[a_1,a_2] \in \myAA$.

\smallskip


(b) As stated, the notion of a greedy element depends on the choice of an initial cluster $\{x_1,x_2\}$, so strictly speaking we
should have included something like ``greedy with respect to $\{x_1,x_2\}$." 
However we will show (see Theorem~\ref{main theorem}(d)  below) that the family of greedy elements is independent of this choice.
\end{remark}

The following proposition provides a motivation for the concept of greedy elements, and also for the term ``greedy."

\begin{proposition}
\label{pr:inequality-c(p,q)}  
Suppose  $x \in \myAA$ is pointed at $(a_1, a_2) \in \ZZ^2$, and is positive at three consecutive clusters $\{x_0,x_1\}$, 
$\{x_1,x_2\}$, and $\{x_2,x_3\}$. 
Then for every nonzero pair of indices $(p,q) \in \ZZ_{\geq 0}^2$, we have the following inequality:
\begin{equation}
\label{eq:c-recurrence inequality} 
\aligned
c(p,q)\ge\max &\left( \sum_{k=1}^p (-1)^{k-1}
c(p-k,q) \binom{a_2\!-\!cq\!+\!k\!-\!1}{k}, \right.\\
&\left.\quad \sum_{k=1}^q
(-1)^{k-1} c(p,q-k) \binom{a_1\!-\!bp\!+\!k\!-\!1}{k}\right)\ .\\
\endaligned
\end{equation}
\end{proposition}

As stated in Remark~\ref{rem:greedy} (a), $c(p,q)$ is easy to compute unless both $a_1$ and $a_2$ are positive.
In the latter case one of the difficulties in dealing with the recurrence relation \eqref{eq:c-recurrence} is the fact that its right hand side is the maximum 
of two linear forms.
Our next result shows that  \eqref{eq:c-recurrence} can be sharpened as follows.

\begin{proposition}
\label{pr:B-introduction-linear}
Let $a_1$ and $a_2$ be positive integers.
The rule \eqref{eq:c-recurrence} is equivalent to
\begin{equation}
\label{eq:c-recurrence2} 
c(p,q)= \left\{
\begin{array}[h]{ll}
\displaystyle \sum_{k=1}^p (-1)^{k-1}
c(p-k,q)\binom{a_2\!-\!cq\!+\!k\!-\!1}{k}, & \quad
\mbox{{\rm if} $ca_1q\leq ba_2p;$}
\\
\displaystyle \sum_{k=1}^q (-1)^{k-1}
 c(p,q-k)\binom{a_1\!-\!bp\!+\!k\!-\!1}{k}, & \quad
\mbox{{\rm if} $ca_1q\geq ba_2p$}
\end{array}
\right.
\end{equation}
for every non-zero pair of indices $(p,q) \in \ZZ_{\geq 0}^2$. 
\end{proposition}

The following theorem summarizes our main results about greedy elements. 

\begin{theorem}
\label{main theorem}
{\rm(a)} For each $(a_1,a_2) \in \ZZ^2$, there exists a (unique) greedy element $x[a_1, a_2] \in \myAA$ at $(a_1, a_2)$.  

{\rm(b)} All greedy elements are indecomposable positive elements.

{\rm(c)} The greedy elements $x[a_1,a_2]$ for $(a_1, a_2) \in \ZZ^2$ form a $\ZZ$-basis in $\myAA$, which we refer to as the greedy basis.

{\rm(d)} The greedy basis is independent of the choice of an initial cluster.

 {\rm(e)} The greedy basis contains all cluster monomials.
\end{theorem}

Several of the statements in Theorem~\ref{main theorem} follow from the symmetry considerations. 
Note that the obvious symmetry of  the exchange relations \eqref{eq:clusterrelations} implies that for every $p \in
\ZZ$, there is an involutive automorphism $\sigma_p$ of $\myAA$
acting on cluster variables by a permutation $\sigma_p(x_m) = x_{2p-m}$.
It is easy to see that the group of automorphisms of $\myAA$ generated
by all $\sigma_p$ is a dihedral group generated by $\sigma_1$ and
$\sigma_2$ (this group is finite if $\myAA$ is of finite type, and infinite otherwise).

\begin{proposition}
\label{co:sigma-symmetry} 
The greedy basis is invariant under the
action of all $\sigma_p$. 
Specifically, the automorphisms $\sigma_1$ and
$\sigma_2$ act on greedy elements as follows:
\begin{equation}
\label{eq:sigma-Q} 
\sigma_1(x[a_1, a_2]) = x[a_1, c [a_1]_+ - a_2], \quad 
\sigma_2 (x[a_1, a_2]) = x[b [a_2]_+ - a_1, a_2] 
\end{equation}
for all $(a_1, a_2) \in \ZZ^2$, where we use the standard notation $[a]_+ =  \max (a, 0)$.
\end{proposition}

To illustrate the use of Proposition~\ref{co:sigma-symmetry}, note that it implies 
Theorem~\ref{main theorem}(e). 
Indeed, it is clear that each cluster monomial can be obtained from a 
cluster monomial in $x_1$ and $x_2$ by the action of some $\sigma_p$.
Since every cluster monomial in $x_1$ and $x_2$ is a greedy element (see 
Remark~\ref{rem:greedy} (1)), it follows that all cluster monomials are greedy elements as well. 
In particular, \eqref{eq:sigma-Q} implies that 
\begin{equation}
\label{eq:2-3-quarter}
x[a_1,a_2]= \left\{
\begin{array}[h]{ll}
x_1^{-a_1} x_0^{a_2}=x_1^{-a_1}x_2^{-a_2}(x_1^b+1)^{a_2}, & \quad
\mbox{if $a_1 \leq 0, \, a_2 \geq 0;$}
\\
x_3^{a_1} x_2^{- a_2}=x_1^{-a_1}x_2^{-a_2}(x_2^c+1)^{a_1}, & \quad
\mbox{if $a_1 \geq 0, \, a_2 \leq 0.$}
\end{array}
\right.
\end{equation}

Another immediate consequence of Proposition~\ref{co:sigma-symmetry} is Theorem~\ref{main theorem}(d). 
Indeed, if we replace the initial cluster $\{x_1, x_2\}$ in the definition of greedy elements by any other cluster, the resulting set of greedy elements 
will be obtained from the original one by some $\sigma_p$.

\begin{remark}
It is clear from the definition that $(a_1,a_2)$ is the \emph{denominator vector} of a greedy element $x[a_1,a_2]$ in the sense of \cite[Section~7]{ca4}.
As a consequence of \cite[Theorem~6.1]{fz-ClusterI}, this observation combined with Theorem~\ref{main theorem}(e) allows us to identify the vectors $(a_1,a_2)$ associated with cluster variables with certain roots. 
To this end, we identify $\ZZ^2$ with the root lattice corresponding to the generalized Cartan matrix 
\begin{equation}
\label{eq:cartan-rank2} A=A(b,c)=\mat{2}{-b}{-c}{2} 
\end{equation}
in such a way that standard basis vectors in $\ZZ^2$ are identified with simple roots.
Then  $x[a_1,a_2]$ is a non-initial cluster variable if and only if $(a_1,a_2)$ is a real positive root under this identification. 
Furthermore, the description of real and imaginary roots given in \cite{kac} implies at once 
that $x[a_1,a_2]$ is a cluster monomial if and only if $(a_1, a_2)$ is \emph{not} an imaginary positive root.

The correspondence between non-initial cluster variables and real positive roots is easily established 
in the finite type case $bc \leq 3$.
To make it explicit in the infinite type case,  
let $S_{-1}(t), S_0 (t), S_1(t), \dots$ be the sequence of
(normalized) \emph{Chebyshev polynomials of second kind} given by
the initial conditions
\begin{equation}
\label{eq:Cheb-2-initial}
S_{-1}(t) = 0, \quad S_0(t) = 1 \, ,
\end{equation}
and the recurrence relation
\begin{equation}
\label{eq:Cheb-2-recurrence}
S_p(t) = t S_{p-1}(t) -  S_{p-2}(t) \quad (p \geq 1) \, .
\end{equation}
Assume that $bc \geq 4$, and let $t = bc - 2$. 
A direct check shows that, for every $m \in \ZZ - \{1,2\}$, we have $x_m = x[a_1,a_2]$, where $(a_1,a_2) \in \ZZ^2$ is given by
\begin{equation}
\label{eq:cluster-variable-denoms}
(a_1,a_2)= \left\{
\begin{array}[h]{llll}
(S_p(t) + S_{p-1}(t), c S_{p-1}(t)), & \quad
\mbox{if $m=2p+3;$}
\\
(b S_{p}(t), S_p(t) + S_{p-1}(t)), & \quad
\mbox{if $m=2p+4;$}
\\
(b S_{p-1}(t), S_p(t) + S_{p-1}(t)), & \quad
\mbox{if $m=-2p;$}
\\
(S_p(t) + S_{p-1}(t), c S_{p}(t)), & \quad
\mbox{if $m=-2p-1;$}
\end{array}
\right.
\end{equation}
here in all the cases we have $p \geq 0$. 
\end{remark}

\smallskip

As already mentioned above, a crucial ingredient in the proofs of the above results is 
an explicit combinatorial expression for the greedy elements, in the spirit of combinatorial expressions for cluster variables given in 
\cite{ls-comm, ls-noncomm, r-noncomm}. 
This expression is given in terms of Dyck paths. 
Here is the necessary terminology. 

Let $(a_1, a_2)$ be a pair of nonnegative integers. 
A \emph{Dyck path} of type $a_1\times a_2$ is a lattice path
from $(0, 0)$  to $(a_1,a_2)$ that 
never goes above the main diagonal joining $(0,0)$ and $(a_1,a_2)$.
Among the Dyck paths of a given type $a_1\times a_2$, there is a (unique) \emph{maximal} one denoted by 
$\mathcal{D} = \mathcal{D}^{a_1\times a_2}$. 
It is defined by the property that any lattice point strictly above $\mathcal{D}$ is also strictly above the main diagonal.

Let $\mathcal{D}=\mathcal{D}^{a_1\times a_2}$.  Let $\mathcal{D}_1=\{u_1,\dots,u_{a_1}\}$ be the set of horizontal edges of $\mathcal{D}$ indexed from left to right, and $\mathcal{D}_2=\{v_1,\dots, v_{a_2}\}$ the set of vertical edges of $\mathcal{D}$ indexed from bottom to top.  
Given any points $A$ and $B$ on $\mathcal{D}$, let $AB$ be the subpath starting from $A$, and going in the Northeast direction until it reaches $B$ (if we reach $(a_1,a_2)$ first, we continue from $(0,0)$). By convention, if $A=B$, then $AA$ is the subpath that starts from $A$, then passes $(a_1,a_2)$ and ends at $A$. If we represent a subpath of $\mathcal{D}$ by its set of edges, then for $A=(i,j)$ and $B=(i',j')$, we have
$$
AB= 
\begin{cases}
\{u_k, v_\ell: i < k \leq i', j < \ell \leq j'\}, \quad\textrm{if $B$ is to the Northeast of $A$};\\
\mathcal{D} - \{u_k, v_\ell: i' < k \leq i, j' < \ell \leq j\}, \quad\textrm{otherwise}.
\end{cases}
$$
We denote by $(AB)_1$ the set of horizontal edges in $AB$, and by $(AB)_2$ the set of vertical edges in $AB$. 
Also let $AB^\circ$ denote the set of lattice points on the subpath $AB$ excluding the endpoints $A$ and $B$ (here $(0,0)$ and $(a_1,a_2)$ are regarded as the same point).

Here is an example for $(a_1,a_2)=(6,4)$. 

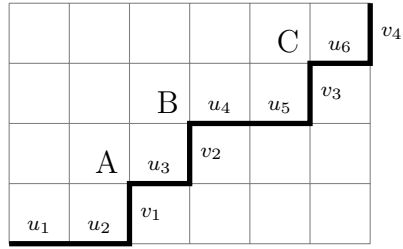
\begin{figure}[h]
\begin{tikzpicture}[scale=.8]
\draw[step=1,color=gray] (0,0) grid (6,4);
\draw[line width=2,color=black] (0,0)--(2,0)--(2,1)--(3,1)--(3,2)--(5,2)--(5,3)--(6,3)--(6,4);
%
\draw (0.5,0) node[anchor=south]  {\tiny$u_1$};
\draw (1.5,0) node[anchor=south]  {\tiny$u_2$};
\draw (2.5,1) node[anchor=south]  {\tiny$u_3$};
\draw (3.5,2) node[anchor=south]  {\tiny$u_4$};
\draw (4.5,2) node[anchor=south]  {\tiny$u_5$};
\draw (5.5,3) node[anchor=south]  {\tiny$u_6$};
\draw (6,3.5) node[anchor=west]  {\tiny$v_4$};
\draw (5,2.5) node[anchor=west]  {\tiny$v_3$};
\draw (3,1.5) node[anchor=west]  {\tiny$v_2$};
\draw (2,.5) node[anchor=west]  {\tiny$v_1$};
\draw (2,1) node[anchor=south east] {A};
\draw (3,2) node[anchor=south east] {B};
\draw (5,3) node[anchor=south east] {C};
\end{tikzpicture}
\caption{A maximal Dyck path.}
\label{fig:Dyck-path}
\end{figure}

Let $A=(2,1)$, $B=(3,2)$ and $C=(5,3)$. Then 
$$
(AB)_1=\{u_3\}, \,\, (AB)_2=\{v_2\}, \,\, 
(BA)_1=\{u_4,u_5,u_6,u_1,u_2\}, \, \, (BA)_2=\{v_3,v_4,v_1\} \ . 
$$
The point $C$ is in $BA^\circ$ but not in $AB^\circ$. The subpath $AA$ has length 10 (not 0).

\begin{definition}
\label{df:compatible}
For $S_1\subseteq \mathcal{D}_1$, $S_2\subseteq \mathcal{D}_2$, we say that the pair $(S_1,S_2)$ is compatible if for every $u\in S_1$ and $v\in S_2$, denoting by $E$ the left endpoint of $u$ and $F$ the upper endpoint of $v$, there exists a lattice point $A\in EF^\circ$ such that 
\begin{equation}
\label{0407df:comp}
|(AF)_1|=b|(AF)_2\cap S_2|\textrm{\; or\; }|(EA)_2|=c|(EA)_1\cap S_1|.\end{equation}
\end{definition}

With all this terminology in place we are ready to present our combinatorial expression for greedy elements. 

\begin{theorem}
\label{th:greedy-combinatorial}
For every $(a_1,a_2) \in \ZZ^2$, the greedy element $x[a_1,a_2] \in \myAA(b,c)$ at $(a_1,a_2)$ is given by
\begin{equation}
\label{eq:greedy-Dyck-expression}
x[a_1,a_2] = x_1^{-a_1}x_2^{-a_2}\sum_{(S_1,S_2)}x_1^{b|S_2|}x_2^{c|S_1|},
\end{equation}
where the sum is over all compatible pairs $(S_1,S_2)$ in $\mathcal{D}^{[a_1]_+\times [a_2]_+}$.
\end{theorem}

\begin{example}
Let $b=3$, $c=2$, and $(a_1,a_2) = (3,3)$. 
Then the Dyck path $\mathcal{D}=\mathcal{D}^{3\times3}$ consists of alternating horizontal and vertical edges: 
$\mathcal{D}= \{u_1, v_1, u_2, v_2, u_3, v_3\}$. 
Here are all compatible pairs $(S_1,S_2)$:
\begin{itemize}
\item[(i)] At least one of the sets $S_1$ and $S_2$ is empty; then another one can be arbitrary.  
\item[(ii)] Both $S_1$ and $S_2$ are non-empty. One can show that there are three such compatible pairs: 
$(\{u_2\},\{v_1\})$, $(\{u_3\},\{v_2\})$, and $(\{u_1\},\{v_3\})$. 
\end{itemize}
For instance, let us show that $(\{u_3\},\{v_2\})$ is compatible. 
We need to check \eqref{0407df:comp} for $E = F = (2,2)$. 
By inspection, the first equality in \eqref{0407df:comp} is impossible to satisfy, but the second one is satisfied for 
$A = (1,1)$ or $A = (2,1)$. 

Adding up the contributions from all these compatible pairs to the right hand side of \eqref{eq:greedy-Dyck-expression}, we see that this formula yields
$${x}[3,3]=x_1^{-3}x_2^{-3}\big{(} (1+x_2^2)^3+ ((1+x_1^3)^3 - 1) + 3x_1^3 x_2^2) \big{)}.$$
\end{example}

\begin{remark}
As a special case of \eqref{eq:greedy-Dyck-expression}, we get a new combinatorial expression for cluster variables, which is different from (and we believe simpler than) the expression given in \cite{ls-comm, ls-noncomm, r-noncomm}.
A combinatorial argument showing the equivalence of these expressions will be given elsewhere. 
\end{remark}

Our proofs of the above results proceed in the following logical sequence. 
Proposition~\ref{pr:inequality-c(p,q)} is proved in Section \ref{inequality-c(p,q)}. 
Then we change our point of view and think of \eqref{eq:greedy-Dyck-expression}
as the \emph{definition} of $x[a_1,a_2]$. 
Clearly, if $x[a_1,a_2]$ is defined this way then it is a Laurent polynomial in $x_1$ and $x_2$ which is pointed at $(a_1,a_2)$. 
The coefficients in its expansion \eqref{eq:pointed-expansion} are given as follows:
\begin{align}
\label{eq:cpq-dyck}
&\text{$c(p,q)$ is the number of compatible pairs $(S_1,S_2)$}\\ 
\nonumber
&\text{in $\mathcal{D}^{[a_1]_+\times [a_2]_+}$ 
such that $|S_1|=q$ and $|S_2|=p$.}
\end{align}

We deduce our main results from the following two technical statements:
\begin{align}
\label{eq:cpq-dyck-symmetry}
&\text{the elements $x[a_1,a_2]$ given by \eqref{eq:greedy-Dyck-expression}}\\ 
\nonumber
&\text{satisfy the symmetry property
\eqref{eq:sigma-Q};}
\end{align}
\begin{align}
\label{eq:cpq-dyck-recurrences}
&\text{if $a_1$ and $a_2$ are positive, then the coefficients $c(p,q)$ }\\ 
\nonumber
&\text{given by \eqref{eq:cpq-dyck} satisfy the recurrence \eqref{eq:c-recurrence2}.}
\end{align}
Property \eqref{eq:cpq-dyck-symmetry} is proved in Section~\ref{1st technical statement}, and \eqref{eq:cpq-dyck-recurrences}
in Section~\ref{2nd technical statement}. 

Once these two properties are established, almost all of the above results (with the exception of Theorem~\ref{main theorem}(c)) can be deduced by the
following sequence of steps.

\smallskip

{\bf Step~1.} In view of \eqref{eq:upper}, the property \eqref{eq:cpq-dyck-symmetry} implies that all $x[a_1,a_2]$ given by \eqref{eq:greedy-Dyck-expression}
do belong to $\myAA$. Furthermore, \eqref{eq:cpq-dyck} makes it obvious that all these elements are positive. 

\smallskip

{\bf Step~2.} We then show that the coefficients
$c(p,q)$ given by \eqref{eq:cpq-dyck} satisfy \eqref{eq:c-recurrence}. 
If at least one of the components $a_1$ and $a_2$ is nonpositive, this follows by a direct check, otherwise we just combine
\eqref{eq:cpq-dyck-recurrences} with Proposition~\ref{pr:inequality-c(p,q)}. 
Thus the elements given by \eqref{eq:greedy-Dyck-expression} are indeed greedy elements in the sense of 
Definition~\ref{df:greedy}. 
This proves  Theorem~\ref{main theorem}(a) and Theorem~\ref{th:greedy-combinatorial}.
We see then that \eqref{eq:cpq-dyck-symmetry} implies Proposition~\ref{co:sigma-symmetry}, while  \eqref{eq:cpq-dyck-recurrences} implies Proposition~\ref{pr:B-introduction-linear}. 

\smallskip

{\bf Step~3.} Once we know that the greedy elements are positive, their indecomposability is a trivial consequence of  
Proposition~\ref{pr:inequality-c(p,q)}.
This completes the proof of Theorem~\ref{main theorem}(b). 

\smallskip

We have already noticed that Proposition~\ref{co:sigma-symmetry} implies Theorem~\ref{main theorem}(d),(e). 
This only leaves Theorem~\ref{main theorem}(c), which will be proved in Section \ref{greedy basis}.
Note that our proof of Theorem~\ref{main theorem}(c) is inspired by a recent paper \cite{bz-triangular}. 

\smallskip

Our proof of \eqref{eq:cpq-dyck-recurrences} uses upper bounds for the \emph{supports} of greedy elements which we obtain in Section~\ref{Upper bounds for supports}
(as usual, the support  of a Laurent polynomial $x \in \ZZ[x_1^{\pm 1}, x_2^{\pm 1}]$ is the set of lattice points $(d_1,d_2)$ such that $x_1^{d_1} x_2^{d_2}$ appears with non-zero coefficient in the Laurent expansion of~$x$). 
The main result in this section is Proposition~\ref{prop:support}. 
The study of these upper bounds brought us to the heuristic conclusion that for general $b$ and $c$ the greedy elements do \emph{not} exhaust 
all indecomposable positive elements in $\myAA(b,c)$.
For instance, our experiments suggest that for $(b,c) = (3,3)$, the element $x[4,7] + x[7,4] - x[1,1]$ is positive, which easily implies  
the existence of a non-greedy  indecomposable positive element in $\myAA(3,3)$; but at the moment we are unable to confirm this decisively.

\section{Proof of Proposition~\ref{pr:inequality-c(p,q)}}
\label{inequality-c(p,q)}

By symmetry it is enough to prove the inequality
\begin{equation}
\label{eq:c-recurrence-2nd-inequality} 
c(p,q)\ge\sum_{k=1}^q (-1)^{k-1}
c(p,q-k) \binom{a_1\!-\!bp\!+\!k\!-\!1}{k}.
\end{equation}
If $bp\ge a_1$, then \eqref{eq:c-recurrence-2nd-inequality} trivially holds since its right hand side is $0$, and $c(p,q)$ is nonnegative by the assumption.
Thus we assume that $bp < a_1$. 
Under this assumption, let $d(p,q)$ denote the difference  between the left hand side and the right hand side of \eqref{eq:c-recurrence-2nd-inequality}, that is, 
$$d(p,q)=
\sum_{k=0}^{q} (-1)^{k}
c(p,q-k) \binom{a_1\!-\!bp\!+\!k\!-\!1}{k} \ .
$$
Thus, Proposition~\ref{pr:inequality-c(p,q)} is immediate from the following lemma. 

\begin{lemma}
\label{lemma:d-as-coefficient} 
Under the assumptions of Proposition~\ref{pr:inequality-c(p,q)}, if $bp < a_1$ then $d(p,q)$ is the coefficient of the monomial $x_2^{-a_2 + cq}x_3^{a_1-bp}$ in the Laurent expansion of $x$ with respect to $\{x_2,x_3\}$. Therefore, $d(p,q) \geq 0$. 
\end{lemma}

\begin{proof}
Take the expansion \eqref{eq:pointed-expansion} of~$x$ and substitute $x_3^{-1} (x_2^c + 1)$ for $x_1$, to get an expansion of $x$ in terms of $x_2$ and $x_3$. 
For a given $p \geq 0$ such that $bp < a_1$, the monomial $x_3^{a_1-bp}$ appears with the coefficient
$$(x_2^c + 1)^{-a_1 + bp} \sum_{q \geq 0} c(p,q)  x_2^{-a_2 + cq} \ .$$
Expanding $(x_2^c + 1)^{-a_1 + bp}$ by the binomial formula with a negative exponent, we get
$$(x_2^c + 1)^{-a_1 + bp} = \sum_{k \geq 0} (-1)^k \binom{a_1\!-\!bp\!+\!k\!-\!1}{k} x_2^{ck} \ .$$
Multiplying this expression with $\sum_{q \geq 0} c(p,q)  x_2^{-a_2 + cq}$, we get 
$\sum_{q \geq 0} d(p,q)  x_2^{-a_2 + cq}$,
finishing the proofs of Lemma~\ref{lemma:d-as-coefficient} and Proposition~\ref{pr:inequality-c(p,q)}. 
\end{proof}

\section{Proof of \eqref{eq:cpq-dyck-symmetry}}
\label{1st technical statement}
Let $x[a_1,a_2]$ be given by \eqref{eq:greedy-Dyck-expression}. 
Due to obvious symmetry, it is enough to prove the second equality in \eqref{eq:sigma-Q}: 
$\sigma_2 (x[a_1, a_2]) = x[a'_1, a_2]$, where $a'_1 = b [a_2]_+ - a_1$. 
Since $\sigma_2$ is an involution, we may also assume that $a_1\leq a'_1$. 
It is easy to see that this assumption shows that it is enough to consider the following three cases: $\max (a_1,a_2) \le 0$, $a_1 \leq 0 < a_2$, $0<a_1<ba_2$ 
(in fact, the last case can be replaced by a stronger restriction $0<a_1\le ba_2/2$, but this does not seem to make our argument easier).
In each of the cases we abbreviate $\mathcal{D} = \mathcal{D}^{[a_1]_+\times [a_2]_+}$ and 
$\mathcal{D}' = \mathcal{D}^{[a'_1]_+\times [a_2]_+}$.

\medskip

{\bf Case 1:} $\max(a_1,a_2) \le 0$. Then $\mathcal{D} = \mathcal{D}^{0 \times 0}$ is just one point so 
$x[a_1,a_2] = x_1^{-a_1} x_2^{-a_2} = x_1^{a'_1} x_2^{-a_2}$. 
Therefore, we have 
$$\sigma_2 (x[a_1, a_2]) = x_3^{a'_1} x_2^{-a_2} = x_1^{-a'_1} x_2^{-a_2} (x_2^c + 1)^{a'_1} \ .$$
On the other hand, $\mathcal{D}' = \mathcal{D}^{a'_1 \times 0}$ is a horizontal segment of length $a'_1$.
Thus, in a compatible pair $(S_1,S_2)$ in $\mathcal{D}'$, the set $S_2$ is empty, while $S_1$ can be any subset of the set of $a'_1$ horizontal edges. 
Applying \eqref{eq:greedy-Dyck-expression} we get
$$x[a'_1, a_2] = x_1^{-a'_1} x_2^{-a_2} (x_2^c + 1)^{a'_1} = \sigma_2 (x[a_1, a_2]) \ ,$$
as desired. 

 \medskip
 
 Before treating the remaining two cases, we make the following easy observation. 
 
 \begin{lemma}
 \label{lem:vj}
 Suppose $a_1$ and $a_2$ are positive integers, and, for $j = 1, \dots, a_2$, let $v_j \in \mathcal{D} = \mathcal{D}^{a_1\times a_2}$ 
 be the vertical edge with the upper endpoint $F_j$ of height~$j$. Then the horizontal coordinate of $F_j$ is $\left  \lceil j a_1/a_2 \right \rceil$, and so the horizontal distance $|(F_h F_j)_1|$ 
 between $F_h$ and $F_j$ for $0\le h< j\le a_2$ is equal to
 \begin{equation}
 \label{eq:ej}
 |(F_{h} F_j)_1| = 
 \left \lceil j a_1/a_2 \right \rceil - \left \lceil h a_1/a_2 \right  \rceil
 \end{equation}
 (with the convention that $F_0$ is the origin $(0,0)$). 
 \end{lemma}
 
 \medskip

{\bf Case 2:} $a_1 \leq 0 < a_2$. Then we have $a'_1 = ba_2 - a_1 \geq b a_2 > 0$. 
The same argument as in Case~1 above shows that $x[a_1, a_2] = x_1^{-a_1} x_2^{-a_2} (x_1^b + 1)^{a_2}$ implying that
\begin{equation}
\label{eq:sigma-2-explicit}
\sigma_2 (x[a_1, a_2]) = x_1^{- a'_1} x_2^{-a_2} (x_2^c + 1)^{a'_1 - ba_2} ((x_2^c + 1)^b + x_1^b)^{a_2} \ .
\end{equation}
Comparing this with the expression \eqref{eq:greedy-Dyck-expression} applied to $\mathcal{D}' = \mathcal{D}^{a'_1 \times a_2}$, we need to prove that 
$$(x_2^c + 1)^{a'_1 - ba_2} ((x_2^c + 1)^b + x_1^b)^{a_2}=\sum_{(S'_1,S'_2)}x_1^{b|S'_2|}x_2^{c|S'_1|} \ ,$$
where the sum on the right is over all compatible pairs in $\mathcal{D}'$. 
Comparing the coefficients of $x_1^{bp'}$ on both sides for $0\le p'\le a_2$, 
it is enough to show that 
$$\binom{a_2}{p'}(x_2^c + 1)^{a'_1 - p'} = \sum_{(S'_1,S'_2): \, |S'_2|=p'} x_2^{c|S'_1|} \ .$$
Letting $X=x_2^c$ and noticing that there are $\binom{a_2}{p'}$ ways to choose $S'_2$, we see that  \eqref{eq:sigma-2-explicit}
becomes a consequence of  the following identity: 
\begin{equation}
\label{eq:case2-algebraic}
\sum_{S'_1} X^{|S'_1|} = (X+1)^{a'_1 - b |S'_2|} 
\end{equation}
for each subset $S'_2\subseteq \mathcal{D}'_2$, where the sum is over all $S'_1 \subseteq \mathcal{D}'_1$ such that $(S'_1,S'_2)$ is compatible.

To prove \eqref{eq:case2-algebraic}, it suffices to prove the following combinatorial statement.

\begin{lemma}
\label{lem:shadow-case2}
Suppose $0 < ba_2 \leq a'_1$.
Then for every $S'_2 \subseteq \mathcal{D}'_2$ there exists a subset $\SH(S'_2)\subseteq \mathcal{D}'_1$ of cardinality $|\SH(S'_2)|= b|S'_2|$
such that, for a subset $S'_1 \subseteq \mathcal{D}'_1$, the pair $(S'_1,S'_2)$ is compatible if and only if $S'_1 \cap \SH(S'_2) = \emptyset$.     
\end{lemma}

\begin{proof}
In view of \eqref{eq:ej}, if  $a'_1 \geq b a_2 > 0$ then the horizontal distance between any two consecutive vertical edges on  $\mathcal{D}'$ is at least~$b$. 
We define $\SH(S'_2)$ as the set that contains the $b$ horizontal edges preceding each vertical edge from $S'_2$.
Thus $|\SH(S'_2)|= b|S'_2|$ as desired. 
We call $\SH(S'_2)$  the \emph{shadow} of $S'_2$ (hence the notation). 

First we assume that $S'_1\cap \SH(S'_2) \neq \emptyset$. 
Take $u\in S'_1\cap \SH(S'_2)$, and let~$v$ be the first vertical edge after~$u$ (moving as always in the Northeast direction). Note that $v$ is in $S'_2$ by the definition of $\SH(S'_2)$. 
By inspection, this pair of edges does \emph{not} satisfy \eqref{0407df:comp}, making the pair $(S'_1,S'_2)$ incompatible.

It remains to show the converse statement: if $S'_1\cap \SH(S'_2) = \emptyset$ then $(S'_1,S'_2)$ is compatible. 
To see that every two edges $u \in  S'_1$ and $v \in S'_2$ satisfy the first case in \eqref{0407df:comp}, we 
can choose a lattice point~$A$ as the left endpoint of the $b$-th horizontal edge in $\mathcal{D}'$ that precedes~$v$. 
\end{proof}

{\bf Case 3:}  $0 < a_1 < ba_2$. 
Recall that $a'_1 = ba_2 - a_1$, so $(a'_1,a_2)$ also falls into this case.  
Again a little algebraic reasoning shows that the desired equality $\sigma_2 (x[a_1, a_2]) = x[a'_1, a_2]$
is implied by the following statement.

\begin{lemma}
\label{lem:case3-algebraic}
Suppose $0 < a_1 < ba_2$. 
There exists a bijection $S_2 \mapsto S'_2$ from subsets of $\mathcal{D}_2$ to subsets of $\mathcal{D}'_2$ such that,
for every $S_2 \subseteq \mathcal{D}_2$, we have  $|S'_2| = a_2 - |S_2|$, and  
\begin{equation}
\label{eq:case3}
\sum_{S_1} X^{|S_1|} = (X+1)^{a_1-b|S_2|}  \sum_{S'_1} X^{|S'_1|} \ ,
\end{equation}
where the first sum is over all $S_1 \subseteq \mathcal{D}_1$ such that $(S_1,S_2)$ is compatible, while the second sum is over all 
$S'_1 \subseteq \mathcal{D}'_1$ such that $(S'_1,S'_2)$ is compatible.
\end{lemma}

Recall that, for $j = 1, \dots, a_2$, we denote by $v_j \in \mathcal{D}_2$ the vertical edge with the upper endpoint $F_j$ of height~$j$.
Let $v'_j$ have the same meaning for the Dyck path $\mathcal{D}'$. 
We define the bijection $S_2 \mapsto S'_2$ in Lemma~\ref{lem:case3-algebraic} by setting  
\begin{equation}
\label{S'2}
S'_2=\{v'_j: v_{a_2+1-j} \in \mathcal{D}_2 - S_2\}.
\end{equation}
The equality $|S'_2| = a_2 - |S_2|$ is trivial, thus to prove Lemma~\ref{lem:case3-algebraic} it remains to prove \eqref{eq:case3}. 

We deduce \eqref{eq:case3} from the next two combinatorial lemmas. 
The first of them is an analogue of Lemma~\ref{lem:shadow-case2}. 

\begin{lemma}
\label{lemma:N}
Suppose $0 < a_1 < ba_2$.
Then for every $S_2 \subseteq \mathcal{D}_2$ there exist two subsets $\sh(S_2)\subseteq \SH(S_2)\subseteq \mathcal{D}_1$ with the following properties:
\begin{enumerate}
\item $|\SH(S_2)| = \min(a_1, b |S_2|)$;

\item For a subset $S_1 \subseteq \mathcal{D}_1$, the pair $(S_1,S_2)$ is compatible if and only if\\ 
$S_1 \cap (\SH(S_2) - \sh(S_2)) = \emptyset$, and $(S_1 \cap \sh(S_2), S_2)$ is compatible.
\end{enumerate}     
\end{lemma}

As in Case~2, we call $\SH(S_2)$ the \emph{shadow} of $S_2$; and we refer to $\sh(S_2)$ as the \emph{remote shadow} of $S_2$
(hence the notation). 
Note that since $(a'_1, a_2)$ also falls into our current case, Lemma~\ref{lemma:N} is also applicable to the subset 
$S'_2$ given by \eqref{S'2}. 

For any $S_2 \subseteq \mathcal{D}_2$ we denote 
$$\mathcal{T}(S_2) = \{S_1 \subseteq \sh(S_2) : (S_1,S_2) \textrm{ is compatible}\} \ .$$

\begin{lemma}
\label{lemma:T}
Suppose $0 < a_1 < ba_2$.
Then for every $S_2 \subseteq \mathcal{D}_2$ there exists a bijection $\theta: \sh(S_2) \to \sh(S'_2)$
that induces a bijection $\mathcal{T}(S_2) \to \mathcal{T}(S'_2)$ (denoted by the same letter $\theta$ with some abuse of notation).  
In particular, we have $|\theta(S_1)|=|S_1|$ for all $S_1\in\mathcal{T}(S_2)$. 
\end{lemma}
\medskip

Before proving Lemmas~\ref{lemma:N} and \ref{lemma:T}, we show that they indeed imply \eqref{eq:case3}.
First of all, by Lemma~\ref{lemma:N}(1) we have $|\mathcal{D}_1-\SH(S_2)| = a_1 - \min(a_1, b |S_2|) = [a_1 - b |S_2|]_+$. 
By the same token, we have 
$$|\mathcal{D}'_1-\SH(S'_2)| = [a'_1 - b |S'_2|]_+ = [-a_1 + b |S_2|]_+ = [a_1 - b |S_2|]_+ - (a_1-b|S_2|)$$
(for the last equality note that $[a]_+ - a = [-a]_+$ for all $a \in \ZZ$). 

Using Lemma~\ref{lemma:N}(2) to split up $S_1$ into a portion outside $\SH(S_2)$ and a portion inside $\sh(S_2)$, 
the left-hand side of \eqref{eq:case3} can be expressed as
$$(X+1)^{[a_1 - b |S_2|]_+} \sum_{S_1\in\mathcal{T}(S_2)} X^{|S_1|}.$$
Analogously, the right-hand side of \eqref{eq:case3} can be expressed as
$$(X+1)^{[a_1 - b |S_2|]_+} \sum_{S'_1\in\mathcal{T}(S'_2)} X^{|S'_1|} \ .$$
By Lemma~\ref{lemma:T}, these two expressions are equal to each other, finishing the proof of \eqref{eq:case3}.   
\qed

\medskip

Now we turn to the proof of Lemma~\ref{lemma:N}.
Our first task is to define the shadow $\SH(S_2)$ and the remote shadow $\sh(S_2)$ for every subset $S_2 \subseteq \mathcal{D}_2$. 

\begin{definition}
\label{df:shadow}
For every vertical edge $v\in S_2\subseteq\mathcal{D}_2$ with the upper endpoint $F$, let $\SH(v;S_2)$ be the set of horizontal edges $(AF)_1$
in the shortest subpath $AF$ of $\mathcal{D}$ such that $|(AF)_1|=b|(AF)_2\cap S_2|$. 
If there is no such a subpath, we define $\SH(v;S_2)$ as  $(FF)_1 = \mathcal{D}_1$.
We call $\SH(v;S_2)$ the local shadow of $S_2$ at $v$. 
We define the shadow of $S_2$ by setting $\SH(S_2)=\cup_{v\in S_2}\SH(v;S_2)$. 
\end{definition}

The definition of $\sh(S_2)$ requires a little preparation. 
Note that Definition~\ref{df:shadow} implies at once that, unless $\SH(v;S_2) = \mathcal{D}_1$, we have $|\SH(v;S_2)| \geq b$. 
On the other hand, if $v = v_j$ then in view of \eqref{eq:ej}, we have $|(F_{j-1} F_j)_1| \leq b$. 
We conclude that the local shadow $\SH(v_j;S_2)$ always contains all the horizontal edges in $\mathcal{D}_1$ of height $j-1$. 
This puts the following definition on the firm ground.

\begin{definition}
\label{df:remote-shadow}
For every $S_2 \subseteq\mathcal{D}_2$ the remote shadow $\sh(S_2)$ is obtained from $\SH(S_2)$ by removing, for each
$v_j \in S_2$,  all the horizontal edges in $\mathcal{D}_1$ of height $j-1$. 
\end{definition}

\begin{example} 
We illustrate the above definions and statements with the following example. 
Let $a_1=13$, $a_2=8$, $b=4$, and $S_2=\{v_2,v_6,v_8\}$. Then $a'_1=19$, and $S'_2=\{v'_2,v'_4,v'_5,v'_6,v'_8\}$.
The shadows of various kinds related to $S_2$ are shown in the left part of Figure~\ref{fig:shadows}, while those related to $S'_2$ are shown in the right part (using the same conventions). 
The vertical edges in $S_2$ are drawn in dotted lines. 
The three local shadows $\SH(v_2;S_2)$, $\SH(v_6;S_2)$ and $\SH(v_8;S_2)$ are the projections of the respective grey strips to $\mathcal{D}_1$.
Thus we have  $\SH(S_2)=\mathcal{D}_1-\{u_5\}$. 
The edges in $\SH(S_2) - \sh(S_2)$ are those immediately below the  shaded strips.
The edges in the remote shadow $\sh(S_2)$ are drawn as dotted horizontal edges; they are labeled \circled{1} -- \circled{8}. 
The map $\theta$ (to be defined later) sends each circled edge in the left part of Figure~\ref{fig:shadows} to the edge with the same label in the right part.

\begin{figure}[h]
\begin{tikzpicture}[scale=.4]
\draw[step=1,color=gray!30] (0,0) grid (13,8);
\draw[line width=.5,color=black] (0,0)--(2,0)--(2,1)--(4,1)--(4,2)--(5,2)--(5,3)--(7,3)--(7,4)--(9,4)--(9,5)--(10,5)--(10,6)--(12,6)--(12,7)--(13,7)--(13,8);
\fill[black!20] (0,1) rectangle (4,2-.1);\draw[line width=1.2, dotted](4,1)--(4,2);
\fill[black!20] (6,5) rectangle (10,6-.1);\draw[line width=1.2, dotted](10,5)--(10,6);
\fill[black!20] (5,7) rectangle (13,8-.1);\draw[line width=1.2, dotted](13,7)--(13,8);
\draw[line width=1.2, dotted](0,0)--(2,0);
\draw[line width=1.2, dotted](5,3)--(7,3);
\draw[line width=1.2, dotted](7,4)--(9,4);
\draw[line width=1.2, dotted](10,6)--(12,6);
\draw(0.5,0.2) node[anchor=north]{\Tinycircled{1}};
\draw(1.5,0.2) node[anchor=north]{\Tinycircled{2}};
\draw(5.5,3.2) node[anchor=north]{\Tinycircled{3}};
\draw(6.5,3.2) node[anchor=north]{\Tinycircled{4}};
\draw(7.5,4.2) node[anchor=north]{\Tinycircled{5}};
\draw(8.5,4.2) node[anchor=north]{\Tinycircled{6}};
\draw(10.5,6.2) node[anchor=north]{\Tinycircled{7}};
\draw(11.5,6.2) node[anchor=north]{\Tinycircled{8}};
\begin{scope}[shift={(16,0)}]
\draw[step=1,color=gray!30] (0,0) grid (19,8);
\fill[black!20] (1,1) rectangle (5,2-.1);\draw[line width=1.2, dotted](5,1)--(5,2);
\fill[black!20] (6,3) rectangle (10,4-.1);\draw[line width=1.2,dotted](10,3)--(10,4);
\fill[black!20] (0,4) rectangle (12,5-.1);\draw[line width=1.2,dotted](12,4)--(12,5);
\fill[black!20] (0,5) rectangle (19,6-.1);\draw[line width=1.2,dotted](15,5)--(15,6);
\fill[black!20] (15,7) rectangle (19,8-.1);\draw[line width=1.2,dotted](19,7)--(19,8);
\draw[line width=.5,color=black] (0,0)--(3,0)--(3,1)--(5,1)--(5,2)--(8,2)--(8,3)--(10,3)--(10,4)--(12,4)--(12,5)--(15,5)--(15,6)--(17,6)--(17,7)--(19,7)--(19,8);
\draw[line width=1.2,dotted](0,0)--(3,0);
\draw[line width=1.2,dotted](5,2)--(8,2);
\draw[line width=1.2,dotted](15,6)--(17,6);
\draw(15.5,6.2) node[anchor=north]{\Tinycircled{1}};
\draw(16.5,6.2) node[anchor=north]{\Tinycircled{2}};
\draw(0.5,.2) node[anchor=north]{\Tinycircled{3}};
\draw(5.5,2.2) node[anchor=north]{\Tinycircled{4}};
\draw(6.5,2.2) node[anchor=north]{\Tinycircled{5}};
\draw(7.5,2.2) node[anchor=north]{\Tinycircled{6}};
\draw(1.5,.2) node[anchor=north]{\Tinycircled{7}};
\draw(2.5,.2) node[anchor=north]{\Tinycircled{8}};
\end{scope}
\end{tikzpicture}
\caption{Shadows.}
\label{fig:shadows}
\end{figure}

\end{example}

\medskip

Our next goal is to prove Lemma~\ref{lemma:N}(2). 
A look at the definitions~\ref{df:compatible} and~\ref{df:shadow} makes it clear that the property that $(S_1, S_2)$ is compatible is not affected 
by adding to or removing from a subset $S_1 \subseteq \mathcal{D}_1$ any subset of $\mathcal{D}_1 - \SH(S_2)$. 
Thus in proving Lemma~\ref{lemma:N}(2) we can assume that $S_1 \subseteq \SH(S_2)$. 
An easy inspection shows that if $S_1$ contains a horizontal edge $u$ of height $j-1$ then $u$ and $v = v_j$ cannot satisfy \eqref{0407df:comp}. 
Thus, if $(S_1, S_2)$ is compatible then $S_1 \cap (\SH(S_2) - \sh(S_2)) = \emptyset$, finishing the proof of Lemma~\ref{lemma:N}(2).

\medskip

To complete the proof of Lemma~\ref{lemma:N} it remains to show the equality $|\SH(S_2)| = \min(a_1, b |S_2|)$. 
We start with the following observation (recall that the notation $AB^\circ$ stands for  the set of interior lattice points of a subpath $AB$ of $\mathcal{D}$, i.e., it is obtained from $AB$ by removing the endpoints $A$ and $B$).

\begin{lemma}
\label{lem:compatibility and inequality}
Suppose $v\in S_2\subseteq\mathcal{D}_2$, and let $\SH(v;S_2) = (AF)_1$ be as in Definition~\ref{df:shadow}.
Then we have $|(A'F)_1|<b|(A'F)_2\cap S_2|$ for every $A'\in AF^\circ$.  
\end{lemma}

\begin{proof}
Let $f(A') = b|(A'F)_2\cap S_2| - |(A'F)_1|$. 
If $A'$ is the lower endpoint of $v$ then $f(A') = b > 0$. 
Now let us move this point away from $F$ (in the Southwest direction) one edge at a time. 
Clearly, at each step the value of $f(A')$ either increases by $b$, stays constant, or decreases by $1$. 
It follows that $f(A')$ remains positive until it first reaches the value~$0$. 
This completes the proof. 
\end{proof}

We need one more lemma to finish the proof of Lemma~\ref{lemma:N}(1).  

\begin{lemma}
\label{lem:disjoint_contain}
If $v$ and $v'$ are distinct vertical edges from $S_2$, and both local shadows $\SH(v;S_2)$ and $\SH(v';S_2)$ 
are different from $\mathcal{D}_1$,  then either these local shadows are disjoint, or one of them is a proper
subset of another. 
\end{lemma}

\begin{proof}
Let $\SH(v;S_2) = (AF)_1$ and $\SH(v';S_2) = (A'F')_1$ in accordance with Definition~\ref{df:shadow}. 
It suffices to show that the lattice paths $AF$ and $A'F'$ cannot overlap, i.e., that it is impossible 
to have $A \in {A'F'}^\circ$ and $F' \in AF^\circ$. 
Indeed if these inclusions were true, by Lemma~\ref{lem:compatibility and inequality} we would have
$$|(AF')_1|<b|(AF')_2\cap S_2|, \quad |(F'F)_1|<b|(F'F)_2\cap S_2| \ .$$
Adding up these two inequalities yields $|(AF)_1|<b|(AF)_2\cap S_2|$, contradicting the definition of 
$\SH(v;S_2)$. 
\end{proof}

Now everything is ready for a proof of the desired equality $|\SH(S_2)| = \min(a_1, b |S_2|)$. 
It follows easily from the next two claims:
\begin{align}
\label{eq:local-shadows1}
&\text{If, for a given $S_2$, all local shadows $\SH(v;S_2)$ are proper subsets}\\
\nonumber
& \text{of $\mathcal{D}_1$ then $|\SH(S_2)| = b |S_2|$; 
in particular, in this case $b|S_2| \leq a_1=|\mathcal{D}_1|$.}\\ 
\label{eq:local-shadows2}
&\text{If $b|S_2| < a_1$ then all local shadows $\SH(v;S_2)$ are proper subsets of $\mathcal{D}_1$.} 
\end{align}

\smallskip

\noindent {\it Proof of \eqref{eq:local-shadows1}.} In view of Lemma~\ref{lem:disjoint_contain}, the shadow $\SH(S_2)$ is the disjoint union of
\emph{maximal} local shadows $\SH(v;S_2)$ (those not contained in another local shadow). 
A maximal local shadow $\SH(v;S_2)$ has cardinality $|\SH(v;S_2)| = b |\{v' \in S_2: \SH(v';S_2) \subseteq \SH(v;S_2)\}|$.
Adding up these cardinalities, we conclude that $|\SH(S_2)| = b |S_2|$, as claimed. 

\smallskip

\noindent {\it Proof of \eqref{eq:local-shadows2}.} Let $v$ be a vertical edge in $S_2$ with the upper endpoint~$F$, and let 
the local shadow $\SH(v;S_2)$ be expressed as usual: $\SH(v;S_2) = (AF)_1$. 
We need to show that $A \neq F$.  
Consider a lattice point $A' \in \mathcal{D}$ such that $|(A'F)_1| = b |S_2|$. 
In view of the assumption $b|S_2| < a_1$, we have $A' \neq F$.
Since $|(A'F)_1| \geq b|(A'F)_2\cap S_2|$, Lemma~\ref{lem:compatibility and inequality} implies that $A'$ does not belong to 
$AF^\circ$. 
Therefore, $A \neq F$, finishing the proofs of \eqref{eq:local-shadows2} and of Lemma~\ref{lemma:N}. 
\qed

\medskip

Now we turn to the proof of Lemma~\ref{lemma:T}.
To construct a desired bijection $\theta: \sh(S_2) \to \sh(S'_2)$, we break the remote shadow
$\sh(S_2)$ into the disjoint union of pieces $\sh(S_2)_{h;j}$ defined as follows.

\begin{definition}
\label{def:remote-pieces}
Let $h$ and $j$ be integers such that $0 \leq h < a_2$, and $0 < j \leq a_2$. 
We denote by $\sh(S_2)_{h;j}$ the set of horizontal edges $u$ of height
$h$ in $\mathcal{D}_1$ such that $u \in \SH(v_j;S_2)$, and $v_j$ is the
first edge after $u$ with this property (that is, the path $E F_j$ is shortest possible, where 
$E$ is the left endpoint of $u$). 
\end{definition}

Clearly, each piece $\sh(S_2)_{h;j}$ is the set of edges in some horizontal interval 
in $\mathcal{D}_1$, and the remote shadow $\sh(S_2)$ is indeed the disjoint union of pieces $\sh(S_2)_{h;j}$. 
Also $\sh(S_2)_{h;j}$ is empty unless $v_j \in S_2$ and $v_{h+1} \in \mathcal{D}_2 - S_2$.

\begin{lemma}
\label{lem:remote-pieces-equal}
For any $h$ and $j$ as in Definition~\ref{def:remote-pieces}, we have
$|\sh(S_2)_{h;j}| = |\sh(S'_2)_{a_2-j; a_2 - h}|$.
\end{lemma} 

Lemma~\ref{lem:remote-pieces-equal} allows us to define a desired bijection 
$\theta: \sh(S_2) \to \sh(S'_2)$ as follows:
\begin{align}
\label{eq:def-theta}
&\text{for each $h$ and $j$ as above, $\theta$ sends 
$\sh(S_2)_{h;j}$ onto}\\ 
\nonumber
&\text{$\sh(S'_2)_{a_2-j; a_2 - h}$ preserving the left-to-right order.}
\end{align}
Clearly, $\theta$ is indeed a bijection $\sh(S_2) \to \sh(S'_2)$; furthermore, \eqref{eq:def-theta}
makes it clear that the inverse bijection $\theta^{-1}$ is the map $\theta': \sh(S'_2) \to \sh(S_2)$ 
defined in the same way as $\theta$ but with $S_2$ and $S'_2$ interchanged. 

\medskip

To prove Lemma~\ref{lem:remote-pieces-equal} we introduce the following notation: 
for each $h$ and $j$ such that $0 \leq h < j \leq a_2$, define an integer $f(h,j) = f(h,j;S_2)$ by setting
\begin{equation}
\label{eq:def-fhj}
f(h,j) = b|(F_h F_j)_2 \cap S_2| - |(F_h F_j)_1| ,
\end{equation}
where the notation $F_j$ is from Lemma \ref{lem:vj}. 

The definition implies at once the following useful \emph{additive property}:
\begin{equation}
\label{eq:f-additive}
\text{$f(h,k) + f(k,j) = f(h,j)$ whenever $h < k < j$.}
\end{equation}

The following ``duality relation" is a direct consequence of Lemma~\ref{lem:vj} and the definition of $S'_2$ 
given by \eqref{S'2}. 

\begin{lemma}
\label{lem:f-duality}
For every $h$ and $j$ such that $0 \leq h < j \leq a_2$, we have $f(h,j; S_2) = - f(a_2 - j, a_2 - h;S'_2)$.
\end{lemma}

After this preparation we turn to the proof of Lemma~\ref{lem:remote-pieces-equal}.
It is enough to treat the case where $0 \leq h < j \leq a_2$ (the case where $h \geq j$ can be reduced to this one
by some adjustment of indices caused by the convention that the path $F_h F_j$ passes through $(a_1, a_2)$ and then continues from the origin).
We also assume that $v_j \in S_2$ and $v_{h+1} \in \mathcal{D}_2 - S_2$ (clearly, this condition then also holds if we replace $\mathcal{D}$ with 
$\mathcal{D}'$, and the triple $(h,j,S_2)$ with $(a_2 - j, a_2 - h, S'_2)$). 
In particular, this implies that $j >  h+1$. 

Using Lemma~\ref{lem:f-duality}, we conclude that Lemma~\ref{lem:remote-pieces-equal} is a consequence of the following statement. 

\begin{lemma}
\label{lem:remote-pieces-cardinality}
Suppose $0 \leq h < j \leq a_2$, and $v_j \in S_2, \,\, v_{h+1} \in \mathcal{D}_2 - S_2$.
Then $\sh(S_2)_{h;j} \neq \emptyset$ if and only if we have 
\begin{equation}
\label{eq:criterion-piece-nonempty}
\text{$f(h,k) < 0 < f(k,j)$ whenever $h < k < j$ }\ .
\end{equation}
Furthermore, if \eqref{eq:criterion-piece-nonempty} is satisfied then 
\begin{equation}
\label{eq:piece-cardinality}
|\sh(S_2)_{h;j}|  = \displaystyle{\min_{h < k < j} \min(f(k,j), -f(h,k))}\ .
\end{equation}
\end{lemma}

\begin{proof}
We start with the following observation. 
Let $u \in \mathcal{D}_1$ be a horizontal edge of height $h$ with the left endpoint $E$. 
As an easy consequence of Lemma~\ref{lem:compatibility and inequality}, $u$ belongs to the local shadow $\SH(v_j; S_2)$ if and only if 
we have 
\begin{equation}
\label{eq:criterion-local-shadow}
\text{$|(E F_j)_1| \leq b|(F_h F_j)_2\cap S_2|$, and  $0 < f(k,j)$ whenever $h < k < j$.}
\end{equation}
In particular, the last condition in \eqref{eq:criterion-local-shadow} is \emph{necessary} for $\sh(S_2)_{h;j} \neq \emptyset$.

Next we show that if $f(h,k) \geq 0$ for some $k$ with $h < k < j$ then $\sh(S_2)_{h;j} = \emptyset$. 
Indeed, let $\ell$ be the smallest integer such that $h < \ell < j$, and $f(h, \ell) \geq 0$. 
If $\ell = h+1$ then $\mathcal{D}_1$ has no edges of height $h$, so the equality $\sh(S_2)_{h;j} = \emptyset$ is trivial. 
Thus, we assume that $\ell > h+1$. 
By the choice of $\ell$, for every $k$ such that $h < k < \ell$, we have $f(h,k) < 0$.
The additive property  \eqref{eq:f-additive} then implies that $f(k,\ell) > 0$. 
In particular, we have $f(\ell -1,\ell) > 0$, implying that $v_\ell \in S_2$. 
Now we see that \eqref{eq:criterion-local-shadow} must hold if we replace $j$ with $\ell$, and $E$ with $F_h$. 
But then, as we just proved, every horizontal edge of height $h$ in $\mathcal{D}_1$ must belong to $\SH(v_\ell; S_2)$, implying
that $\sh(S_2)_{h;j} = \emptyset$.

We have shown that the conditions \eqref{eq:criterion-piece-nonempty} are necessary for $\sh(S_2)_{h;j} \neq \emptyset$. 
The fact that they are sufficient follows at once from \eqref{eq:piece-cardinality}.
So we assume that \eqref{eq:criterion-piece-nonempty} is satisfied, and focus on the proof of \eqref{eq:piece-cardinality}. 

Remembering Definition~\ref{df:remote-shadow} and using the criterion \eqref{eq:criterion-local-shadow}, we conclude that 
a lattice point $E \in \mathcal{D}$ is the left endpoint of a horizontal edge that belongs to $\sh(S_2)_{h;j}$ if and only if
it satisfies the following inequalities:
$$\displaystyle{\max_{\ell \in L}}  (b |(F_h F_\ell)_2 \cap S_2| + |(F_\ell F_j)_1|) < |(EF_j)_1| \leq 
\min (b|(F_h F_j)_2\cap S_2|, |(F_h F_j)_1|) \ ,$$
where 
$$L = \{h+1\} \cup \{\ell: h < \ell < j, \,\, v_\ell \in S_2, \,\, f(k,\ell) > 0 \,\, {\rm for} \,\, h < k < \ell\} \ .$$ 
Therefore, we have
\begin{align*}
|\sh(S_2)_{h;j}| &= [\min (b|(F_h F_j)_2\cap S_2|, |(F_h F_j)_1|) - \displaystyle{\max_{\ell \in L}}  (b |(F_h F_\ell)_2 \cap S_2| + |(F_\ell F_j)_1|)]_+\\
& = [\displaystyle{\min_{\ell \in L}} \min(f(\ell,j), -f(h,\ell))]_+ = \displaystyle{\min_{\ell \in L}} \min(f(\ell,j), -f(h,\ell))
\end{align*}
(the last equality is due to \eqref{eq:criterion-piece-nonempty}). 
It remains to show that this expression for $|\sh(S_2)_{h;j}|$ agrees with \eqref{eq:piece-cardinality}. 
By the additive property  \eqref{eq:f-additive}, we have $-f(h,\ell) = f(\ell,j) - f(h,j)$, implying that
$$\displaystyle{\min_{\ell \in L}} \min(f(\ell,j), -f(h,\ell)) = \displaystyle{\min_{\ell \in L}} (f(\ell,j)) - [f(h,j)]_+ \ ;$$ 
Thus it suffices to show the following:
$$\displaystyle{\min_{\ell \in L}} (f(\ell,j)) = \displaystyle{\min_{h < k < j}} (f(k,j)) \ .$$
Let $\ell$ be the smallest value of $k$ that attains the minimum $\min_{h < k < j} (f(k,j))$. 
An argument parallel to the one used in the second paragraph of the proof then shows that $\ell \in L$, finishing the proofs of Lemma~\ref{lem:remote-pieces-cardinality} and Lemma~\ref{lem:remote-pieces-equal}. 
 \end{proof}
 
We have already noted that Lemma~\ref{lem:remote-pieces-equal} makes a bijection 
$\theta: \sh(S_2) \to \sh(S'_2)$ well-defined via \eqref{eq:def-theta}. 
To finish the proof of Lemma~\ref{lemma:T} it suffices to prove the following.

\begin{lemma}
\label{prop:theta}
Let $S_2$ be a subset of $\mathcal{D}_2$, and $S'_2$ be given by \eqref{S'2}.
Suppose a subset $S_1$ of $\sh(S_2)$ is such that $(S_1,S_2)$ is not compatible.
Then $(\theta(S_1), S'_2)$ is also not compatible. 
\end{lemma}

\begin{proof}
By the definition, there exist $u \in S_1$ and $v \in S_2$ not satisfying \eqref{0407df:comp}. 
Looking at the first case in \eqref{0407df:comp}, we may assume without loss of generality that $v = v_j$, and $u \in \sh(S_2)_{h;j}$ for some index~$h$. 
As in the proof of Lemma~\ref{lem:remote-pieces-equal}, it is enough to treat the case where $0 \leq h < j \leq a_2$. 
Then the failure of the second case in \eqref{0407df:comp} can be expressed as follows: for every $k$ such that $h < k < j$, we have
\begin{equation}
\label{eq:failure-second-case}
c |(E F_k)_1 \cap S_1| > k - h \ ,
\end{equation} 
where $E$ is the left endpoint of $u$ (this follows by the same argument as in the proof of Lemma~\ref{lem:compatibility and inequality}). 
Clearly, we can assume that $u$ is the leftmost edge in $\sh(S_2)_{h;j}$ (this makes \eqref{eq:failure-second-case} only easier to satisfy). 
In view of Lemma~\ref{lem:disjoint_contain}, we see that \eqref{eq:failure-second-case} is equivalent to the following system of inequalities:
\begin{equation}
\label{eq:failure-second-case-thru-shadows}
\displaystyle{\sum_{(h',j'): h' < k}} g(h',j') > k - h \quad (h < k < j) \ ,
\end{equation}
where we abbreviate $g(h',j') = c |\sh(S_2)_{h';j'} \cap S_1|$ (with the convention that all the indices run over the fixed interval $[h,j]$, and that $g(h',j') = 0$ 
unless $\sh(S_2)_{h';j'} \neq \emptyset$, so that in particular we must have $j' > h'+1$).    

In particular, setting $k = j-1$ specializes \eqref{eq:failure-second-case-thru-shadows} to
\begin{equation}
\label{eq:failure-second-case-thru-shadows-total}
\displaystyle{\sum_{(h',j')}} g(h',j') > j - h - 1  \ .
\end{equation}
We claim that \eqref{eq:failure-second-case-thru-shadows-total} implies the following property:
\begin{eqnarray}
\label{eq:failure-dual}
&\text{there exists an index $\ell > h+ 1$ such that, for every $k$}\\
\nonumber
&\text{with $h < k < \ell$, we have $\displaystyle{\sum_{(h',j'): k < j' \leq \ell}} g(h',j') > \ell - k$}
\end{eqnarray}
(recall that we are still using the convention that all indices belong to $[h,j]$). 
Assume for the sake of contradiction that \eqref{eq:failure-dual} does not hold, that is, for every $\ell > h+1$ there exists
an index $k$ such that $h < k < \ell$, and 
\begin{equation}
\label{eq:failure-dual-contradiction}
\displaystyle{\sum_{(h',j'): k < j' \leq \ell}} g(h',j') \leq \ell - k \ .
\end{equation}
First we use \eqref{eq:failure-dual-contradiction} for $\ell = \ell_0 = j$, and define $\ell_1$ as any of the possible values of $k$. 
If $\ell_1 > h+1$, then we use \eqref{eq:failure-dual-contradiction}  for $\ell = \ell_1$, and again define $\ell_2$ as any of the possible values of $k$.
We continue in the same way, generating the sequence $j-1 = \ell_0 > \ell_1 > \cdots > \ell_r$, that terminates at $\ell_r = h+1$.
Adding up all the inequalities  \eqref{eq:failure-dual-contradiction} used along the way, we get
$$\displaystyle{\sum_{(h',j')}} g(h',j') \leq \ell_0 - \ell_r = j - h -1 \, $$
in contradiction to \eqref{eq:failure-second-case-thru-shadows-total}.

Choose an index $\ell$ satisfying \eqref{eq:failure-dual}. 
Without loss of generality we can assume that $g(h, \ell) > 0$ (otherwise replace the interval $[h,\ell]$ with its maximal by inclusion 
subinterval $[h',j']$ such that $g(h',j') > 0$). 
Now recall the definition \eqref{eq:def-theta} of the map~$\theta$, which allows us to express $g(h',j')$ as
$$g(h',j') = c |\sh(S_2)_{h';j'} \cap S_1| = c |\sh(S'_2)_{a_2-j';a_2-h'} \cap \theta(S_1)| \ .$$
Substituting these expressions into \eqref{eq:failure-dual}, we see that this system of inequalities becomes identical to the system of the kind 
\eqref{eq:failure-second-case-thru-shadows} with the quadruple $(h,j,S_1, S_2)$ replaced by $(a_2 - \ell, a_2 - h, \theta(S_1), S'_2)$. 
It follows that $(\theta(S_1), S'_2)$ is not compatible, finishing the proof of Lemma~\ref{prop:theta}. 
As we have seen, this also completes the proofs of Lemma~\ref{lemma:T} and of the last case in the proof of \eqref{eq:cpq-dyck-symmetry}.
\end{proof}

\section{Upper bounds for supports of greedy elements}
\label{Upper bounds for supports}

Recall that the \emph{support}  of a Laurent polynomial $x \in \ZZ[x_1^{\pm 1}, x_2^{\pm 1}]$ is the set of lattice points $(d_1,d_2) \in \ZZ^2$ such that $x_1^{d_1} x_2^{d_2}$ appears with non-zero coefficient in the Laurent expansion of~$x$. 
In this section we obtain upper bounds for the supports of all greedy elements $x[a_1,a_2]$. 
Since $x[a_1,a_2]$ is pointed at $(a_1, a_2) \in \ZZ^2$ (see Definition~\ref{df:pointed}), i.e., has the expansion
\eqref{eq:pointed-expansion}, we find it more convenient to work with the set $\{(p,q) \in \ZZ_{\geq 0}^2 : c(p,q) \neq 0\}$. 
We refer to this set as the \emph{pointed support} of $x[a_1,a_2]$ and denote it by $PS[a_1,a_2]$; thus $PS[a_1,a_2]$ is the support of the polynomial 
$X[a_1,a_2]  \in \ZZ[X_1, X_2]$ such that $x[a_1,a_2](x_1,x_2) = x_1^{-a_1} x_2^{-a_2} X[a_1,a_2](x_1^b, x_2^c)$. 
Knowing the pointed support we recover the ordinary support as follows:
\begin{align}
\label{eq:pointed-ordinary-support}
&\text{The support of $x[a_1,a_2]$ is the image of its pointed support}\\
\nonumber
 &\text{under the affine map $(p,q) \mapsto (-a_1 + bp, -a_2 + cq)$.}
\end{align}  
The following proposition provides an upper bound for $PS[a_1,a_2]$.  
It involves six cases covering all $(a_1,a_2) \in \ZZ^2$.

\begin{proposition}
\label{prop:support}
\begin{enumerate}
\item If $a_1 \leq 0$ and $a_2 \leq 0$ then $PS[a_1,a_2] = \{(0,0)\}$.

\item If $a_1 \leq 0 < a_2$ then $PS[a_1,a_2] = \{(p,0): 0 \leq p \leq a_2\}$, that is, 
$PS[a_1,a_2]$ is the set of lattice points in the closed segment with vertices $(0,0)$ and $(a_2,0)$. 

\item If $a_2 \leq 0 < a_1$ then $PS[a_1,a_2] = \{(0,q): 0 \leq q \leq a_1\}$, that is, 
$PS[a_1,a_2]$ is the set of lattice points in the closed segment with vertices $(0,0)$ and $(0,a_1)$. 

\item If $a_1 \geq ba_2 > 0$ then $PS[a_1,a_2]$ is contained in the set of lattice points in the closed trapezoid 
with vertices $(0,0)$, $(a_2,0)$, $(a_2, a_1 - ba_2)$, and $(0,a_1)$. 

\item If $a_2 \geq ca_1 > 0$ then $PS[a_1,a_2]$ is contained in the set of lattice points in the closed trapezoid 
with vertices $(0,0)$, $(a_2,0)$, $(a_2 - ca_1, a_1)$, and $(0,a_1)$. 

\item If $0 < a_1 < ba_2$, and $0 < a_2 < ca_1$ then $PS[a_1,a_2]$ is contained in the set of lattice points in the region bounded by the broken line 
$$(0,0), \,\,(a_2,0), \,\,(a_1/b, a_2/c), \,\, (0,a_1), \,\, (0,0),$$ 
with the convention that this region includes the closed segments $[(0,0), (a_2,0)]$ and $[(0,a_1),(0,0)]$ but excludes the rest of the boundary.  
\end{enumerate}
\end{proposition}

The following figure illustrates cases (4) - (6) in Proposition~\ref{prop:support}. 

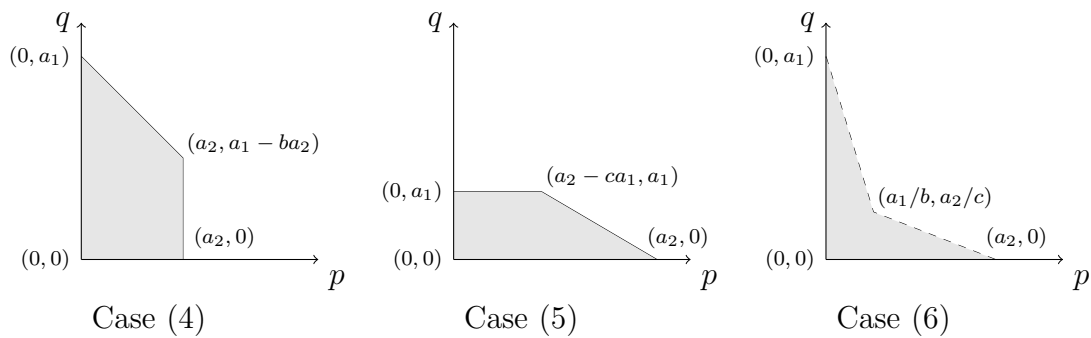
\begin{figure}[h]
\begin{tikzpicture}[scale=.9]
\usetikzlibrary{patterns}
\draw (0,3)--(1.5,1.5)--(1.5,0);
\fill [black!10] (0,3)--(1.5,1.5)--(1.5,0)--(0,0)--(0,3);
\draw (0,0) node[anchor=east] {\tiny$(0,0)$};
\draw (1.5,0) node[anchor=south west] {\tiny$(a_2,0)$};
\draw (1.4,1.4) node[anchor= south west] {\tiny$(a_2,a_1-ba_2)$};
\draw (0,3) node[anchor=east] {\tiny$(0,a_1)$};
\draw[->] (0,0) -- (3.5,0)
node[below right] {$p$};
\draw[->] (0,0) -- (0,3.5)
node[left] {$q$};
\draw (1,-.5) node[anchor=north] {Case (4)};
\begin{scope}[shift={(5.5,0)}]
\usetikzlibrary{patterns}
\draw (3,0)--(1.3,1)--(0,1);
\fill [black!10] (3,0)--(1.3,1)--(0,1)--(0,0)--(3,0);
\draw (0,0) node[anchor=east] {\tiny$(0,0)$};
\draw (2.7,0) node[anchor=south west] {\tiny$(a_2,0)$};
\draw (1.2,.9) node[anchor=south west] {\tiny$(a_2-ca_1,a_1)$};
\draw (0,1) node[anchor=east] {\tiny$(0,a_1)$};
\draw[->] (0,0) -- (3.5,0)
node[below right] {$p$};
\draw[->] (0,0) -- (0,3.5)
node[left] {$q$};
\draw (1,-.5) node[anchor=north] {Case (5)};
\end{scope}
\begin{scope}[shift={(11,0)}]
\usetikzlibrary{patterns}
\draw[dashed] (0,3)--(.7,.7)--(2.5,0);
\fill [black!10]  (0,3)--(.7,.7)--(2.5,0)--(0,0)--(0,3);
\draw (0,0) node[anchor=east] {\tiny$(0,0)$};
\draw (2.2,0) node[anchor=south west] {\tiny$(a_2,0)$};
\draw (.6,.6) node[anchor=south west] {\tiny$(a_1/b,a_2/c)$};
\draw (0,3) node[anchor=east] {\tiny$(0,a_1)$};
\draw[->] (0,0) -- (3.5,0)
node[below right] {$p$};
\draw[->] (0,0) -- (0,3.5)
node[left] {$q$};
\draw (1,-.5) node[anchor=north] {Case (6)};
\end{scope}
\end{tikzpicture}
\caption{Cases (4)-(6) of Proposition 4.1.}
\label{fig:Proposition 4.1}
\end{figure}

Note that in the last case the polygonal region in question does not have to be convex, and the vertex $(a_1/b, a_2/c)$ is not necessarily a lattice point.

\medskip

Before proving Proposition~\ref{prop:support}, we note that it has the following useful corollary.

\begin{corollary}
\label{cor:out-of-1st-quadrant}
If at least one of $a_1$ and $a_2$ is positive then the support of $x[a_1,a_2]$ has empty intersection with the positive quadrant $\ZZ_{\geq 0}^2$.
\end{corollary}

This follows by inspection after applying the affine transformation \eqref{eq:pointed-ordinary-support} to the regions in cases~(2) - (6) in 
Proposition~\ref{prop:support}. 

\begin{proof}[Proof of Proposition~\ref{prop:support}]
Case (1) is trivial: we have $x[a_1,a_2] = x_1^{-a_1} x_2^{-a_2}$, hence $X[a_1,a_2] = 1$.

In Case (2) we have $x[a_1,a_2] = x_1^{-a_1} x_0^{a_2} =  x_1^{-a_1} ((x_1^b + 1)/x_2)^{a_2}$, hence 
$X[a_1,a_2] = (X_1 + 1)^{a_2}$, implying the desired statement. 
Case (3) follows from Case (2) by obvious symmetry. 

Now suppose that $(a_1,a_2)$ is as in Case (4). 
Remembering Case~2 in Section~\ref{1st technical statement}, we note that $x[a_1,a_2]$ is given by the right side of 
\eqref{eq:sigma-2-explicit} with $a'_1$ replaced by $a_1$.  
It follows that
\begin{equation}
\label{eq:X-case4-explicit}
X[a_1, a_2] =  (X_2 + 1)^{a_1 - ba_2} ((X_2 + 1)^b + X_1)^{a_2} \ .
\end{equation}
Therefore, the Newton polygon of $X[a_1, a_2]$ (that is, the convex hull of $PS[a_1,a_2]$) is the Minkowski sum 
of the segment $[(0,0), (0, a_1 - ba_2)]$ and the triangle with vertices $(0,0), (0, ba_2), (a_2, 0)$. 
By inspection, this Minkowski sum is exactly the trapezoid described in (4), finishing the proof in this case. 
Case (5) follows from Case (4) by obvious symmetry. 

Our proof of Case (6) is more involved than the previous ones. 
We use the description of the coefficients $c(p,q)$ given by \eqref{eq:cpq-dyck}. 
Thus $PS[a_1,a_2]$ is the set of pairs $(p, q)$ such that there is a compatible pair $(S_1,S_2)$ in $\mathcal{D}^{a_1 \times a_2}$ 
with $|S_1|=q$ and $|S_2|=p$.
This implies in particular that a lattice point $(p,0)$ belongs to $PS[a_1,a_2]$ if and only if $0 \leq p \leq a_2$; and similarly,
a lattice point $(0,q)$ belongs to $PS[a_1,a_2]$ if and only if $0 \leq q \leq a_1$.
To prove the rest of part (6) it is enough to show that every compatible pair $(S_1, S_2)$ satisfies the following three claims: 
\begin{align}
\label{claim1}
&\text{If $0 < b|S_2| < a_1$,  then the lattice point $(|S_2|,|S_1|)$}\\ 
\nonumber
&\text{lies strictly below the segment $[(0,a_1), (a_1/b, a_2/c)]$.}\\
\label{claim2}
&\text{If $0 < c|S_1| < a_2$,  then the lattice point $(|S_2|,|S_1|)$}\\ 
\nonumber
&\text{lies strictly to the left of the segment $[(a_2,0), (a_1/b, a_2/c)]$.}\\
\label{claim3}
&\text{The case where $b|S_2| \geq a_1$, and $c|S_1| \geq a_2$ is impossible.}
\end{align}

\noindent \emph{Proof of \eqref{claim1}:} by a simple calculation, the condition that $(|S_2|,|S_1|)$ 
lies strictly below the segment $[(0,a_1), (a_1/b, a_2/c)]$ is equivalent to the following: 
\begin{equation}
\label{eq:claim1-restate}
|S_1| < a_1 - b|S_2| + \frac{ba_2|S_2|}{ca_1} \ .
\end{equation} 
Recalling Lemma~\ref{lemma:N}, we note that in our case $|\mathcal{D}_1 - \SH(S_2)| =  a_1 - b|S_2|$, hence
\eqref{eq:claim1-restate}  is equivalent to the following: 
\begin{equation}
\label{eq:claim1-restate-2}
\text{If $S_1 \subseteq \sh(S_2)$, and $(S_1,S_2)$ is compatible then $ca_1|S_1| < ba_2|S_2|$.}
\end{equation} 

We start the proof of \eqref{eq:claim1-restate-2} with an observation (to be used in a moment):
\begin{align}
\label{eq:slope}
& \text{If $E$ is the left endpoint of a horizontal edge in $\mathcal{D}^{a_1\times a_2}$, and $F$ is the}\\
\nonumber
& \text{upper endpoint of a vertical edge then
$a_1(|(EF)_2| -1) < a_2|(EF)_1|$.}
\end{align}
To see this, assume that $F$ is at height $j$ and $E$ is at height $h-1$. Then $|(EF)_1| \ge |(F_h F_j)_1|+1$
and $|(EF)_2|-1=|(F_hF_j)|=j-h$.  Thus using  Lemma~\ref{lem:vj}, we have $a_2|(EF)_1|\ge a_2(|F_hF_j|+1)=a_2(\lceil ja_1/a_2\rceil-\lceil ha_1/a_2\rceil+1)>a_2(ja_1/a_2-ha_1/a_2)=a_1(j-h)=a_1(|(EF)_2| -1)$.

Now suppose $S_1$ and $S_2$ are as in \eqref{eq:claim1-restate-2}. 
Recall from Lemma~\ref{lem:disjoint_contain} and \eqref{eq:local-shadows2} that under the assumption $b|S_2| < a_1$,
all local shadows $\SH(v;S_2)$ are proper subsets of $\mathcal{D}_1$, and the shadow $\SH(S_2)$ is the disjoint union of
maximal local shadows $\SH(v;S_2)$ (those not contained in another local shadow).
Let $\SH(v;S_2)$ be one of these maximal local shadows, and let $F$ be the upper endpoint of $v$, and $E$ the left endpoint of the leftmost edge 
in $\SH(v;S_2) \cap S_1$. 
To prove \eqref{eq:claim1-restate-2} it is enough to show that
\begin{equation}
\label{eq:claim1-restate-3}
ca_1|(EF)_1 \cap S_1| < ba_2|(EF)_2 \cap S_2| \ .
\end{equation} 
In view of Lemma~\ref{lem:compatibility and inequality}, we have $|(EF)_1| \leq b |(EF)_2 \cap S_2|$,
so \eqref{eq:claim1-restate-3} reduces to 
\begin{equation}
\label{eq:claim1-restate-4}
ca_1|(EF)_1 \cap S_1| < a_2|(EF)_1| \ .
\end{equation} 
Using \eqref{eq:slope} we see that \eqref{eq:claim1-restate-4} is in turn a consequence of 
\begin{equation}
\label{eq:claim1-restate-5}
c |(EF)_1 \cap S_1| < |(EF)_2| \ .
\end{equation}

We prove \eqref{eq:claim1-restate-5} by means of the following construction. To start we set $E(0) = E$. Since $(S_1, S_2)$ is compatible and the first case of the condition (1.14) cannot be satisfied by the definition of the shadow, there must exist a point $F(0) \in E(0) F^\circ$ such that $F(0)$ is the upper endpoint of a vertical edge in $\mathcal{D}$, and  $c |(E(0)F(0))_1 \cap S_1| = |(E(0) F(0))_2|$. 
If $(F(0) F)_1 \cap S_1 \neq \emptyset$, we denote by $E(1)$ the left endpoint of the leftmost edge in $(F(0) F)_1 \cap S_1$, and then find $F(1) \in E(1) F^\circ$
so that $c |(E(1)F(1))_1 \cap S_1| = |(E(1) F(1))_2|$.
Continuing in the same way, we construct a sequence of pairs $(E(0), F(0)), \dots, (E(r),F(r))$ terminating when $(F(r) F)_1 \cap S_1 = \emptyset$.
As a result we have 
$$c |(EF)_1 \cap S_1| = \sum_{s=0}^r c |(E(s)F(s))_1 \cap S_1| = \sum_{s=0}^r |(E(s)F(s))_2| \leq |(EF(r))_2| < |(EF)_2| \ ,$$
proving \eqref{eq:claim1-restate-5} and completing the proof of \eqref{claim1}.  

\medskip

\noindent \emph{Proof of \eqref{claim2}:} this claim is obtained from \eqref{claim1} by obvious symmetry replacing $(b,c,a_1,a_2,S_1,S_2)$ with 
$(c,b,a_2,a_1,S^T_2,S^T_1)$, where $S^T_2 = \{u_j: v_{a_2 + 1- j} \in S_2\}$,  and $S^T_1 = \{v_i: u_{a_1 + 1- i} \in S_1\}$.

\medskip

\noindent \emph{Proof of \eqref{claim3}:} first consider the case where all the local shadows $\SH(v;S_2)$ are proper subsets of $\mathcal{D}_1$.
Recalling \eqref{eq:local-shadows1}, we conclude that in this case we have $|\SH(S_2)| = b |S_2| = a_1$. 
Now observe that the proof of \eqref{claim1} applies verbatim in this case, and so \eqref{eq:claim1-restate} still holds, yielding
$$c|S_1| < \frac{ba_2|S_2|}{a_1} = a_2 \ ,$$
as desired. 

The symmetry described in the proof of \eqref{claim2} takes care of the case where the assumption of \eqref{eq:local-shadows1} is satisfied after the replacement
of $(b,c,a_1,a_2,S_1,S_2)$ with $(c,b,a_2,a_1,S^T_2,S^T_1)$. 
Thus it remains to consider the case where $\SH(v^\circ;S_2) = \mathcal{D}_1$ for some $v^\circ \in S_2$, and the same condition holds after the above mentioned symmetry.
By the definition of the shadow, the first case of the condition \eqref{0407df:comp} cannot be satisfied for $v = v^\circ$. 
By symmetry, there is also an edge $u^\circ \in S_1$ such that the second case of \eqref{0407df:comp} cannot be satisfied for $u = u^\circ$. 
It follows that the pair $(u^\circ, v^\circ)$ violates \eqref{0407df:comp}, making the pair $(S_1,S_2)$ not compatible, in contradiction to our assumption.
This completes the proof of Proposition~\ref{prop:support}.
\end{proof}

\section{Proof of \eqref{eq:cpq-dyck-recurrences}}
\label{2nd technical statement}

We deduce \eqref{eq:cpq-dyck-recurrences} from the results in Section~\ref{Upper bounds for supports}.
Recall that our goal is to show that, for each positive integers $a_1$ and $a_2$,  the coefficients $c(p,q)$ 
given by \eqref{eq:cpq-dyck}
satisfy the recurrence relations \eqref{eq:c-recurrence2}. 
Our usual symmetry considerations show that it suffices  to prove the second equality in \eqref{eq:c-recurrence2}:
$$c(p,q)= \sum_{k=1}^q (-1)^{k-1}
 c(p,q-k)\binom{a_1\!-\!bp\!+\!k\!-\!1}{k}$$ 
whenever $(p,q) \neq (0,0)$, and $ca_1q \geq b a_2 p$.
We need to consider several cases.

\medskip

\noindent {\bf Case 1:}  Suppose that $a_1 \leq bp$.
Since all binomial coefficients in the right hand side of the desired equality are equal to $0$, we need to show that $c(p,q) = 0$. 
Note that $c(p,q)$ is the coefficient of $x_1^{-a_1 + bp} x_2^{-a_2 + cq}$ in the Laurent expansion of $x[a_1,a_2]$.  
Now observe that the assumptions $a_1 \leq bp$ and $c a_1q \geq ba_2p$ imply that $cq \geq a_2$. 
Thus the lattice point $(-a_1 + bp, -a_2 + cq)$ lies in the positive quadrant $\ZZ_{\geq 0}^2$, and the desired equality $c(p,q) = 0$
follows from Corollary~\ref{cor:out-of-1st-quadrant}. 

\medskip

\noindent {\bf Case 2:} Now suppose that $bp < a_1$.
The difference between $c(p,q)$ and the right hand side of
the second equality in  \eqref{eq:c-recurrence2} is 
$$d(p,q)=\sum_{k=0}^q (-1)^{k}
 c(p,q-k)\binom{a_1\!-\!bp\!+\!k\!-\!1}{k}.$$
By Lemma~\ref{lemma:d-as-coefficient}, $d(p,q)$ is the
coefficient of a $x_2^{-a_2 + cq}x_3^{a_1-bp}$ in the Laurent 
expansion of $x[a_1,a_2]$ with respect to $\{x_2,x_3\}$. 
Applying the automorphism $\sigma_2$ and using \eqref{eq:sigma-Q}, we see that
$d(p,q)$ is also the coefficient
of $x_1^{a_1-bp}x_2^{-a_2 + cq}$ in the expansion of $\sigma_2(x[a_1,a_2])=x[a'_1,a_2]$, where $a'_1=ba_2-a_1$.
In other words, if we denote the coefficients in the expansion \eqref{eq:pointed-expansion} of $x[a'_1,a_2]$ by $c'(p,q)$ then 
we have $d(p,q) = c'(p',q)$, where $p' = a_2 - p$.
Thus we need to show that $c'(p',q) = 0$ under the following conditions obtained by 
expressing our current assumptions on $p,q,a_1,a_2$ in terms of $p',q,a'_1,a_2$:
\begin{align}
\label{eq:p'a'-1}
& a_2 > 0, \,\, ba_2 > a'_1 \ ;\\
\label{eq:p'a'-2}
& (p',q) \in \ZZ_{\geq 0}^2, \,\, (p',q) \neq (a_2,0), \,\, bp' > a'_1 \ ;\\ 
\label{eq:p'a'-3}
& c(ba_2 - a'_1) q \geq ba_2(a_2 - p') \ .
\end{align}

We will use Proposition~\ref{prop:support} to show that, under the assumption \eqref{eq:p'a'-1}, a lattice point $(p',q)$ satisfying
\eqref{eq:p'a'-2} - \eqref{eq:p'a'-3} \emph{cannot} belong to the pointed support $PS[a'_1,a_2]$. 
Our arguments are based on the following key observation: let 
$$L = \{(p',q) \in \RR^2: c(ba_2 - a'_1) q = ba_2(a_2 - p')\} \ ;$$
then $L$ is a straight line with a negative slope passing through the point $(a_2,0)$,
and the condition \eqref{eq:p'a'-3} means that $(p',q)$ lies on or above this straight line. 

Now we have the following three subcases.

\medskip

\noindent {\bf Subcase 2.1:} Suppose that $a'_1 \leq 0$. 
As shown in Case~(2) in Proposition~\ref{prop:support}, the pointed support $PS[a'_1,a_2]$  is the set of lattice points 
in the closed segment $[(0,0), (a_2,0)]$. 
This segment lies below the line~$L$ (with the exception of the point $(a_2,0)$ that belongs to $L$).
Thus $(a_2,0)$ is the only point from $PS[a'_1,a_2]$ that satisfies \eqref{eq:p'a'-3}. 
But this point is excluded by \eqref{eq:p'a'-2}, so we are done. 

\medskip

\noindent {\bf Subcase 2.2:} Suppose that $0 < a'_1 < ba_2$, and $a_2 \geq ca'_1$. 
Then we are in Case (5) of Proposition~\ref{prop:support}. 
To prove our claim it is enough to show that the trapezoid described there (with $a_1$ replaced by $a'_1$) lies on the 
``wrong side" (that is strictly below) of $L$. 
It is enough to show that the vertex $(p',q) = (a_2 - ca'_1, a'_1)$ cannot satisfy \eqref{eq:p'a'-3}. 
But this is clear since substituting these values of $p'$ and $q$ into \eqref{eq:p'a'-3} and simplifying, we get an 
obviously false inequality $c  {a'_1}^2 \leq 0$.

\medskip

\noindent {\bf Subcase 2.3:} Finally suppose that $0 < a'_1 < ba_2$, and  $0 < a_2 < ca'_1$.
Now we are in Case (6) of Proposition~\ref{prop:support}.
Note that substituting $(p',q) = (a'_1/b, a_2/c)$ into \eqref{eq:p'a'-3}, we get an equality.
Thus in this case $L$ is the straight line through the points $(a_2,0)$ and  $(a'_1/b, a_2/c)$.
Recall also that \eqref{eq:p'a'-2} includes the condition $p' > a'_1/b$. 
By Case (6) in Proposition~\ref{prop:support}, we conclude that $(a_2,0)$ is the only point from $PS[a'_1,a_2]$ that has a chance to satisfy
\eqref{eq:p'a'-2} and \eqref{eq:p'a'-3}. 
Since (as in Subcase 2.1) this point is excluded by another condition in \eqref{eq:p'a'-2}, we are done in this case too, 
finishing the proof of \eqref{eq:cpq-dyck-recurrences}. 
\qed

\section{Greedy elements form a basis}
\label{greedy basis}

In this section we prove that the greedy elements $x[a_1, a_2]$ for $(a_1,a_2) \in \ZZ^2$ form a $\ZZ$-basis in the
cluster algebra $\myAA(b,c)$, the last result from Section~\ref{sec:intro} that still remains unproven. 
The main idea of the proof is similar to that in \cite{bz-triangular}: we compare the family of greedy elements with a known
basis in $\myAA(b,c)$ formed by \emph{standard monomials}. 
Specifically, for each  $(a_1,a_2) \in \ZZ^2$ we define an element $z[a_1,a_2] \in \myAA(b,c)$ by setting
$$z[a_1,a_2] = x_0^{[a_2]_+} x_1^{[-a_1]_+} x_2^{[-a_2]_+} x_3^{[a_1]_+} \ .$$
As a special case of \cite[Theorem~1.16]{fz-ClusterIII} we have:
\begin{equation}
\label{eq:standard-monomials-basis}
\text{The elements $z[a_1,a_2]$ for all $(a_1,a_2) \in \ZZ^2$ form a $\ZZ$-basis in $\myAA(b,c)$ .}
\end{equation}

We need just two properties of this basis (the first one is immediate from the definitions, and the second follows at once from 
Remark~\ref{rem:greedy} (a) and \eqref{eq:2-3-quarter}):
\begin{align}
\label{eq:standard-mon-pointed}
&\text{Every element $z[a_1,a_2]$ is pointed at $(a_1,a_2)$ (see Definition~\ref{df:pointed});}\\
\label{eq:standard-agrees-greedy}
&\text{If $(a_1,a_2) \in \ZZ^2 - \ZZ_{>0}^2$ then $z[a_1,a_2] = x[a_1,a_2]$.}
\end{align}

Inspired by \cite{bz-triangular}, we introduce the following partial order on $\ZZ^2$:
\begin{equation}
\label{eq:partial-order}
(b_1,b_2) \prec (a_1,a_2) \Longleftrightarrow [b_1]_+ + [b_2]_+ < [a_1]_+ + [a_2]_+ \ .
\end{equation}

\begin{lemma}
\label{lem:unitriangular}
If $(a_1,a_2) \in \ZZ_{> 0}^2$ then the expansion  of the greedy element $x[a_1,a_2]$ in the basis of standard monomials
is of the form
\begin{equation}
\label{eq:expansion-triangular} 
x[a_1,a_2] = z[a_1,a_2] + \sum_{(b_1,b_2) \prec (a_1,a_2)} u(b_1,b_2; a_1,a_2) z[b_1,b_2] \ ,
\end{equation}
where all the coefficients $u(b_1,b_2; a_1,a_2)$ are integers, and only finitely many of them are nonzero.
\end{lemma}

\begin{proof}
Note that in view of \eqref{eq:standard-monomials-basis}, every greedy element $x[a_1,a_2]$ is a \emph{finite} integral linear combination of 
standard monomials. 
Thus we need to show only the triangular property in \eqref{eq:expansion-triangular}. 

Let $z[b_1,b_2]$ be an element that occurs in the expansion of $x[a_1,a_2]$ with a nonzero coefficient, and has the maximal possible value of 
$[b_1]_+ + [b_2]_+$. 
If $[b_1]_+ + [b_2]_+ > [a_1]_+ + [a_2]_+$ (or if $[b_1]_+ + [b_2]_+ = [a_1]_+ + [a_2]_+$, but $(b_1,b_2) \neq (a_1,a_2)$), 
then  the fact that all our elements are pointed implies that the monomial $x_1^{-b_1} x_2^{-b_2}$ appears in the Laurent expansion of
$z[b_1,b_2]$ but does not appear in $x[a_1,a_2]$ or any other element in its expansion.
Thus this case is impossible, proving our claim.  
\end{proof}

Now everything is ready for proving that the elements $x[a_1,a_2]$ form a $\ZZ$-basis in $\myAA(b,c)$. 
Clearly it is enough to show the following:
\begin{align}
\nonumber
&\text{For every finite subset $S \subset \ZZ^2$ there exists a finite subset $T  \subset \ZZ^2$}\\ 
\label{eq:finite-exhaust}  
&\text{such that $S \subseteq T$, and the families $\{x[a_1,a_2]: (a_1,a_2) \in T\}$}\\
\nonumber
&\text{and $\{z[a_1,a_2]: (a_1,a_2) \in T\}$
have the same linear $\ZZ$-spans.}
\end{align}

So let $S$ be a finite subset of $\ZZ^2$. 
Let $S_+ = S \cap \ZZ_{> 0}^2$ and $S_- = S - S_+ = S \cap (\ZZ^2 - \ZZ_{> 0}^2)$.
In view of \eqref{eq:standard-agrees-greedy}, there is nothing to prove if $S = S_-$, so we assume 
that $S_+ \neq \emptyset$.  
Let $d = \max \{a_1+a_2 \ : (a_1,a_2) \in S_+\}$, and define
$$T_+ = \{(a_1,a_2) \in \ZZ_{>0}^2 \ : a_1 + a_2 \leq d\} \ .$$
Thus we have $S_+ \subseteq T_+$. 

Now consider the expansions \eqref{eq:expansion-triangular}  for all $(a_1,a_2) \in T_+$, and 
define 
$$T_- = S_- \cup \{(b_1,b_2) \in \ZZ^2 - \ZZ_{> 0}^2 \ : u(b_1,b_2; a_1,a_2) \neq 0 \,\, 
{\rm for \,\, some} \,\, (a_1,a_2) \in T_+\} \ .$$
Let $T = T_+ \cup T_-$; by construction, the $\ZZ$-span of $\{x[a_1,a_2]: (a_1,a_2) \in T\}$ is contained in 
the $\ZZ$-span of $\{z[a_1,a_2]: (a_1,a_2) \in T\}$. 
To show that these two spans are equal, it is enough to show that the transition matrix expressing the family 
$\{x[a_1,a_2]: (a_1,a_2) \in T\}$ in terms of the basis $\{z[a_1,a_2]: (a_1,a_2) \in T\}$ is invertible over $\ZZ$. 
To write this matrix explicitly, we need to choose a linear order on $T$; we do this by ordering $T_-$ arbitrarily, 
letting $T_+$ to go before $T_-$, and choosing an order of $T_+$ as an arbitrary linear extension of the partial order 
$``\prec"$.
In view of Lemma~\ref{lem:unitriangular} and \eqref{eq:standard-agrees-greedy}, the matrix in question then has the block-diagonal form
$$\begin{pmatrix} U & 0 \\
                  C & I 
\end{pmatrix} \ ,$$
where $U$ is a triangular matrix with $1$s on the diagonal, and $I$ is the identity matrix. 
Clearly, such a matrix is invertible over $\ZZ$, finishing the proof of Theorem \ref{main theorem} (c).
\qed

\section*{Acknowledgments} 
This paper owes a lot to Paul Sherman. 
Definition~\ref{df:greedy} and Proposition~\ref{pr:inequality-c(p,q)} first appeared in
an unpublished follow-up to \cite{sz-Finite-Affine} (with the same authors), and Proposition~\ref{pr:B-introduction-linear}, Theorem~\ref{main theorem}
and Proposition~\ref{co:sigma-symmetry} were stated there as conjectures. 
This was almost a decade ago when Paul was a Master's student at Northeastern under the guidance of the third author. 
After getting his degree Paul has left academia to pursue other interests. 
If he ever decides to come back to research in mathematics, he is very welcome!
 
We are grateful to Gregg Musiker,
Dylan Rupel and Ralf Schiffler for valuable discussions and to the anonymous referee for a very thorough reading of the paper and many useful comments.


\begin{thebibliography}{99}


\bibitem{fz-ClusterIII}
A.~Berenstein, S.~Fomin and A.~Zelevinsky,
Cluster Algebras III: Upper bounds and double Bruhat cells,
\textsl{Duke Math. J.} \textbf{126} (2005), no. 1, 1--52.

\bibitem{bz-triangular}
A.~Berenstein and A.~Zelevinsky,
Triangular bases in quantum cluster algebras,
\texttt{arXiv:1206.3586}.




\bibitem{fz-Laurent}
S.~Fomin and A.~Zelevinsky, The Laurent phenomenon,
\textsl {Adv. in Applied Math.} \textbf{28} (2002), 119--144.


\bibitem{fz-ClusterI}
S.~Fomin and A.~Zelevinsky, Cluster Algebras I: Foundations,
\textsl{J.\ Amer.\ Math.\ Soc.} \textbf{15} (2002), 497--529.

\bibitem{ca4}
S.~Fomin and A.~Zelevinsky,
Cluster algebras~IV: Coefficients,
\textsl{Comp.\ Math.} \textbf{143} (2007), 112--164.



\bibitem{kac}
V.~Kac, \textsl{Infinite dimensional Lie algebras}, 3rd edition,
Cambridge University Press, 1990.

\bibitem{ls-comm}
K.~Lee and R.~Schiffler, A combinatorial formula for rank 2 cluster variables, \textsl{J. Algebraic Combin.}, to appear.

\bibitem{ls-noncomm}
K.~Lee and R.~Schiffler, Proof of a positivity conjecture of M. Kontsevich on noncommutative cluster variables,  \textsl{Compos. Math.}, to appear.

\bibitem{plamondon}
P.-G.~Plamondon,
Generic bases for cluster algebras from the cluster category, \texttt{arXiv:1111.4431}.

\bibitem{r-noncomm}
D.~Rupel, Proof of the Kontsevich Non-Commutative Cluster Positivity Conjecture, \texttt{arXiv:1201.3426}.

\bibitem{sz-Finite-Affine}
P.~Sherman and A.~Zelevinsky, Positivity and Canonical Bases in
Rank 2 Cluster Algebras of Finite and Affine Types, \textsl{Mosc. Math. J.} \textbf{4} (2004), no. 4, 947--974.

\end{thebibliography}
\end{document}